\documentclass[12pt,reqno]{article}

\usepackage[usenames]{color}
\usepackage{amssymb}
\usepackage{amsthm}
\usepackage{graphicx}

\usepackage{booktabs}

\usepackage{amscd}
\usepackage{ifthen}

\usepackage{ tikz }
\usetikzlibrary{ calc }
\tikzset{
Square/.style = {
inner sep = 0 pt,
minimum width = 8 mm,
minimum height = 8 mm,
draw = black,
fill = none,
align = center
}}

\def\tn{\textnormal}
\newcommand{\ignore}[1]{}
\newcommand{\set}[1]{\left\{#1\right\}}
\def\D{\mathcal{D}}

\def\P{\mathcal{P}}
\def\T{\mathcal{T}}

\def\l{\ell}

\usepackage[colorlinks=true,
linkcolor=webgreen,
filecolor=webbrown,
citecolor=webgreen]{hyperref}

\definecolor{webgreen}{rgb}{0,.5,0}
\definecolor{webbrown}{rgb}{.6,0,0}

\usepackage{color}
\usepackage{fullpage}
\usepackage{float}

\usepackage{graphics,amsmath,amssymb}

\usepackage{amsfonts}
\usepackage{latexsym}
\usepackage{epsf}

\newcommand{\seqnum}[1]{\href{http://oeis.org/#1}{\underline{#1}}}

\usepackage[noblocks]{authblk}

\allowdisplaybreaks

\begin{document}

\theoremstyle{plain}
\newtheorem{theorem}{Theorem}
\newtheorem{corollary}[theorem]{Corollary}
\newtheorem{lemma}[theorem]{Lemma}
\newtheorem{proposition}[theorem]{Proposition}

\theoremstyle{definition}
\newtheorem{definition}[theorem]{Definition}
\newtheorem{example}[theorem]{Example}
\newtheorem{conjecture}[theorem]{Conjecture}
\newtheorem{notation}[theorem]{Notation}
\newtheorem{remark}[theorem]{Remark}

\theoremstyle{remark}

\begin{center}
\vskip 1cm{\LARGE\bf 
Some Properties and \\
Combinatorial Implications of \\
\vskip .1in
 Weighted Small Schr\"oder Numbers
}
\vskip 1cm
\large
Yu Hin (Gary) Au\\
Department of Mathematics and Statistics\\
University of Saskatchewan\\
Saskatoon, SK\\
Canada \\
\href{mailto:au@math.usask.ca}{\tt au@math.usask.ca} \\
\end{center}

\begin{abstract}
The $n^{\tn{th}}$ small Schr\"oder number is $s(n) = \sum_{k \geq 0} s(n,k)$, where $s(n,k)$ denotes the number of plane rooted trees with $n$ leaves and $k$ internal nodes that each has at least two children. In this manuscript, we focus on the weighted small Schr\"oder numbers $s_d(n) = \sum_{k \geq 0} s(n,k) d^k$, where $d$ is an arbitrary fixed real number. We provide recursive and asymptotic formulas for $s_d(n)$, as well as some identities and combinatorial interpretations for these numbers. We also establish connections between $s_d(n)$ and several families of Dyck paths.
\end{abstract}

\section{Introduction}\label{sec11}

\subsection{Small Schr\"oder numbers}

The small Schr\"oder numbers, denoted $s(n)$ for every $n \geq 1$ gives the sequence
\begin{center}
\begin{tabular}{l|rrrrrrrrr}
$n$ & 1 & 2 & 3 & 4 & 5 & 6 & 7 & 8 & $\cdots$ \\
\hline
$s(n)$ & 1 & 1 & 3 & 11 & 45 & 197 & 903 & 4279 & $\cdots$ \\
\end{tabular}
\end{center}
(\seqnum{A001003} in the On-line Encyclopedia of Integer Sequences (OEIS)~\cite{OEIS}). This sequence has been studied extensively and has many combinatorial interpretations (see, for instance,~\cite[Exercise 6.39]{Stanley99}). We list two below:

\begin{itemize}
\item
Let $\T_n$ denote the set of \emph{Schr\"oder trees} with $n$ leaves, which are plane rooted trees where each internal node has at least $2$ children. Then $s(n) = |\T_n|$ for every $n \geq 1$. Figure~\ref{fig1} shows the $s(4) = 11$ Schr\"oder trees with $4$ leaves.

\def\fig1scale{0.29}
\begin{center}
\begin{figure}[h]
\[
\begin{array}{c}

\begin{tikzpicture}
[scale=\fig1scale, yscale=0.8, thick,main node/.style={circle,inner sep=0.5mm,draw,font=\small\sffamily}]
  \node[main node] at (4,6) (0) {};
      \node[main node] at (2,4) (1) {};
          \node[main node] at (3.33,4) (2) {};
          \node[main node] at (4.67,4) (3) {};
    \node[main node] at (6,4) (4) {};

  \path[every node/.style={font=\sffamily}]
    (0) edge (1)
    (0) edge (2)
        (0) edge (3)
        (0) edge (4)
;
\end{tikzpicture}

\qquad
\begin{tikzpicture}
[scale=\fig1scale, yscale=0.8, thick,main node/.style={circle,inner sep=0.5mm,draw,font=\small\sffamily}]
  \node[main node] at (4,6) (0) {};
      \node[main node] at (2,4) (1) {};
          \node[main node] at (4,4) (2) {};
          \node[main node] at (6,4) (3) {};
    \node[main node] at (1,2) (11) {};
        \node[main node] at (3,2) (12) {};

  \path[every node/.style={font=\sffamily}]
    (0) edge (1)
    (0) edge (2)
    (0) edge (3)    
        (1) edge (11)
        (1) edge (12)
        ;
\end{tikzpicture}
\qquad

\begin{tikzpicture}
[scale=\fig1scale, yscale=0.8, thick,main node/.style={circle,inner sep=0.5mm,draw,font=\small\sffamily}]
  \node[main node] at (4,6) (0) {};
      \node[main node] at (2,4) (1) {};
          \node[main node] at (4,4) (2) {};
          \node[main node] at (6,4) (3) {};
    \node[main node] at (3,2) (21) {};
        \node[main node] at (5,2) (22) {};

  \path[every node/.style={font=\sffamily}]
    (0) edge (1)
    (0) edge (2)
    (0) edge (3)    
        (2) edge (21)
        (2) edge (22)
        ;
\end{tikzpicture}
\qquad

\begin{tikzpicture}
[scale=\fig1scale, yscale=0.8, thick,main node/.style={circle,inner sep=0.5mm,draw,font=\small\sffamily}]
  \node[main node] at (4,6) (0) {};
      \node[main node] at (2,4) (1) {};
          \node[main node] at (4,4) (2) {};
          \node[main node] at (6,4) (3) {};
    \node[main node] at (5,2) (31) {};
        \node[main node] at (7,2) (32) {};

  \path[every node/.style={font=\sffamily}]
    (0) edge (1)
    (0) edge (2)
    (0) edge (3)    
        (3) edge (31)
        (3) edge (32)
        ;
\end{tikzpicture}

\\
\\

\begin{tikzpicture}
[scale=\fig1scale, yscale=0.8, thick,main node/.style={circle,inner sep=0.5mm,draw,font=\small\sffamily}]
  \node[main node] at (4,6) (0) {};
      \node[main node] at (2,4) (1) {};
          \node[main node] at (6,4) (2) {};
    \node[main node] at (1,2) (11) {};
        \node[main node] at (2,2) (12) {};
        \node[main node] at (3,2) (13) {};

  \path[every node/.style={font=\sffamily}]
    (0) edge (1)
    (0) edge (2)
        (1) edge (11)
        (1) edge (12)
        (1) edge (13)
        ;
\end{tikzpicture}
\qquad

\begin{tikzpicture}
[scale=\fig1scale, yscale=0.8, thick,main node/.style={circle,inner sep=0.5mm,draw,font=\small\sffamily}]
  \node[main node] at (4,6) (0) {};
      \node[main node] at (2,4) (1) {};
          \node[main node] at (6,4) (2) {};
    \node[main node] at (5,2) (21) {};
        \node[main node] at (6,2) (22) {};
        \node[main node] at (7,2) (23) {};

  \path[every node/.style={font=\sffamily}]
    (0) edge (1)
    (0) edge (2)
        (2) edge (21)
        (2) edge (22)
        (2) edge (23)
        ;
\end{tikzpicture}
\qquad

\begin{tikzpicture}
[scale=\fig1scale, yscale=0.8, thick,main node/.style={circle,inner sep=0.5mm,draw,font=\small\sffamily}]
  \node[main node] at (4,6) (0) {};
      \node[main node] at (2,4) (1) {};
          \node[main node] at (6,4) (2) {};
    \node[main node] at (1,2) (11) {};
        \node[main node] at (3,2) (12) {};
    \node[main node] at (5,2) (21) {};
        \node[main node] at (7,2) (22) {};

  \path[every node/.style={font=\sffamily}]
    (0) edge (1)
    (0) edge (2)
        (1) edge (11)
        (1) edge (12)
        (2) edge (21)
        (2) edge (22)
;
\end{tikzpicture}

\\
\\

\begin{tikzpicture}
[scale=\fig1scale, yscale=0.8, thick,main node/.style={circle,inner sep=0.5mm,draw,font=\small\sffamily}]
  \node[main node] at (4,6) (0) {};
      \node[main node] at (2,4) (1) {};
          \node[main node] at (6,4) (2) {};
    \node[main node] at (1,2) (11) {};
        \node[main node] at (3,2) (12) {};
    \node[main node] at (0.5,0) (111) {};
        \node[main node] at (1.5,0) (112) {};

  \path[every node/.style={font=\sffamily}]
    (0) edge (1)
    (0) edge (2)
        (1) edge (11)
        (1) edge (12)
        (11) edge (111)
        (11) edge (112)
;
\end{tikzpicture}
\qquad

\begin{tikzpicture}
[scale=\fig1scale, yscale=0.8, thick,main node/.style={circle,inner sep=0.5mm,draw,font=\small\sffamily}]
  \node[main node] at (4,6) (0) {};
      \node[main node] at (2,4) (1) {};
          \node[main node] at (6,4) (2) {};
    \node[main node] at (1,2) (11) {};
        \node[main node] at (3,2) (12) {};
    \node[main node] at (2.5,0) (121) {};
        \node[main node] at (2.5,0) (122) {};

  \path[every node/.style={font=\sffamily}]
    (0) edge (1)
    (0) edge (2)
        (1) edge (11)
        (1) edge (12)
        (12) edge (121)
        (12) edge (122)
;
\end{tikzpicture}

\qquad

\begin{tikzpicture}
[scale=\fig1scale, yscale=0.8, thick,main node/.style={circle,inner sep=0.5mm,draw,font=\small\sffamily}]
  \node[main node] at (4,6) (0) {};
      \node[main node] at (2,4) (1) {};
          \node[main node] at (6,4) (2) {};
    \node[main node] at (4.5,0) (211) {};
        \node[main node] at (5.5,0) (212) {};
    \node[main node] at (5,2) (21) {};
        \node[main node] at (7,2) (22) {};

  \path[every node/.style={font=\sffamily}]
    (0) edge (1)
    (0) edge (2)
        (2) edge (21)
        (2) edge (22)
        (21) edge (211)
        (21) edge (212)
;
\end{tikzpicture}
\qquad

\begin{tikzpicture}
[scale=\fig1scale, yscale=0.8, thick,main node/.style={circle,inner sep=0.5mm,draw,font=\small\sffamily}]
  \node[main node] at (4,6) (0) {};
      \node[main node] at (2,4) (1) {};
          \node[main node] at (6,4) (2) {};
    \node[main node] at (6.5,0) (221) {};
        \node[main node] at (7.5,0) (222) {};
    \node[main node] at (5,2) (21) {};
        \node[main node] at (7,2) (22) {};

  \path[every node/.style={font=\sffamily}]
    (0) edge (1)
    (0) edge (2)
        (2) edge (21)
        (2) edge (22)
        (22) edge (221)
        (22) edge (222)
;
\end{tikzpicture}

\end{array}
\]

\caption{The $s(4)= 11$ Schr\"oder trees with $4$ leaves}\label{fig1}
\end{figure}
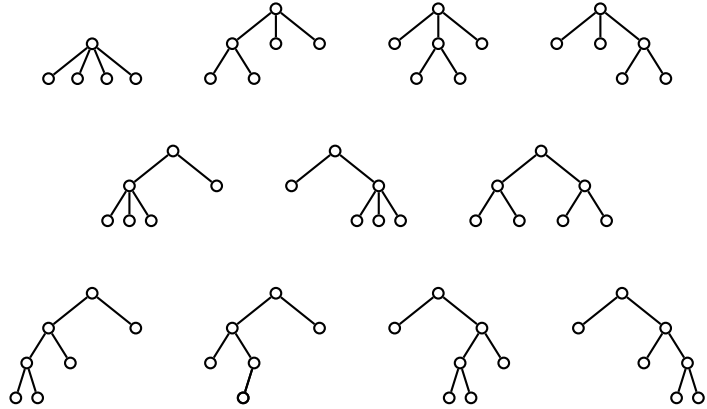
\end{center}

\item
Let $\P_n$ be the set of \emph{small Schr\"oder paths} from $(0,0)$ to $(2n-2,0)$, which are lattice paths that
\begin{itemize}
\item[(P1)]
use only up steps $U = (1,1)$, down steps $D = (1,-1)$, and flat steps $F = (2,0)$;
\item[(P2)]
remain on or above the $x$-axis;
\item[(P3)]
doe not contain an $F$ step on the $x$-axis.
\end{itemize}
Then $s(n) = |\P_n|$ for every $n \geq 1$. Figure~\ref{fig2} lists the $s(4) = 11$ small Schr\"oder paths from $(0,0)$ to $(6,0)$.

\def\fig2scale{0.4}
\begin{center}
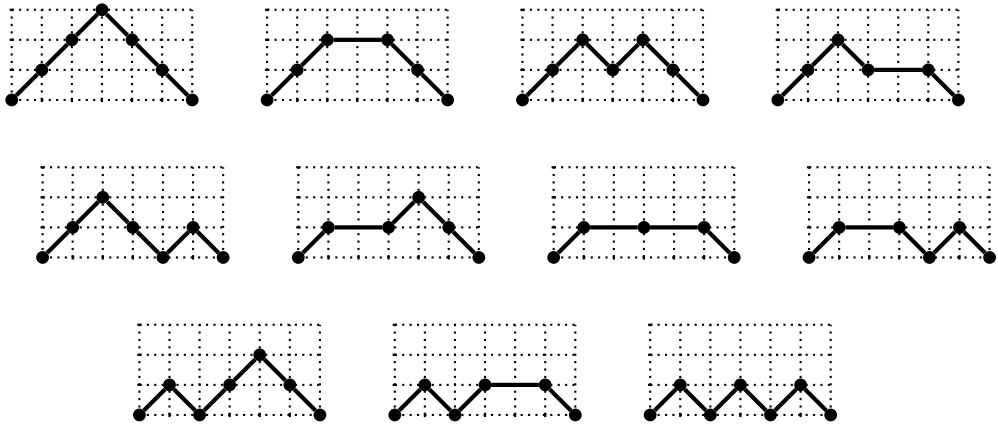
\begin{figure}[h]

\[
\begin{array}{c}
\begin{tikzpicture}[scale = \fig2scale, xscale=1,yscale=1, font=\scriptsize\sffamily, thick,main node/.style={circle,inner sep=0.5mm,draw, fill}]
\def\xlb{0}; \def\xub{6}; \def\ylb{0}; \def\yub{3}; \def\buf{0};
\foreach \x in {\xlb ,...,\xub}
{    \ifthenelse{\NOT 0 = \x}{\draw[thick](\x ,-2pt) -- (\x ,2pt);}{}
\draw[dotted, thick](\x,\ylb- \buf) -- (\x,\yub + \buf);}
\foreach \y in {\ylb ,...,\yub}
{    \ifthenelse{\NOT 0 = \y}{\draw[thick](-2pt, \y) -- (2pt, \y);}{}
\draw[dotted, thick](\xlb- \buf, \y) -- (\xub + \buf, \y);}

\node[main node] at (0,0) (0) {};
\node[main node] at (1,1) (1) {};
\node[main node] at (2,2) (2) {};
\node[main node] at (3,3) (3) {};
\node[main node] at (4,2) (4) {};
\node[main node] at (5,1) (5) {};
\node[main node] at (6,0) (6) {};
\draw[ultra thick] (0) -- (1) -- (2) -- (3)-- (4)-- (5)-- (6);

\end{tikzpicture}
\qquad
\begin{tikzpicture}[scale = \fig2scale, xscale=1,yscale=1, font=\scriptsize\sffamily, thick,main node/.style={circle,inner sep=0.5mm,draw, fill}]
\def\xlb{0}; \def\xub{6}; \def\ylb{0}; \def\yub{3}; \def\buf{0};
\foreach \x in {\xlb ,...,\xub}
{    \ifthenelse{\NOT 0 = \x}{\draw[thick](\x ,-2pt) -- (\x ,2pt);}{}
\draw[dotted, thick](\x,\ylb- \buf) -- (\x,\yub + \buf);}
\foreach \y in {\ylb ,...,\yub}
{    \ifthenelse{\NOT 0 = \y}{\draw[thick](-2pt, \y) -- (2pt, \y);}{}
\draw[dotted, thick](\xlb- \buf, \y) -- (\xub + \buf, \y);}
\node[main node] at (0,0) (0) {};
\node[main node] at (1,1) (1) {};
\node[main node] at (2,2) (2) {};
\node[main node] at (4,2) (3) {};
\node[main node] at (5,1) (4) {};
\node[main node] at (6,0) (5) {};
\draw[ultra thick] (0) -- (1) -- (2) -- (3)-- (4)-- (5);

\end{tikzpicture}
\qquad
\begin{tikzpicture}[scale = \fig2scale, xscale=1,yscale=1, font=\scriptsize\sffamily, thick,main node/.style={circle,inner sep=0.5mm,draw, fill}]
\def\xlb{0}; \def\xub{6}; \def\ylb{0}; \def\yub{3}; \def\buf{0};
\foreach \x in {\xlb ,...,\xub}
{    \ifthenelse{\NOT 0 = \x}{\draw[thick](\x ,-2pt) -- (\x ,2pt);}{}
\draw[dotted, thick](\x,\ylb- \buf) -- (\x,\yub + \buf);}
\foreach \y in {\ylb ,...,\yub}
{    \ifthenelse{\NOT 0 = \y}{\draw[thick](-2pt, \y) -- (2pt, \y);}{}
\draw[dotted, thick](\xlb- \buf, \y) -- (\xub + \buf, \y);}
\node[main node] at (0,0) (0) {};
\node[main node] at (1,1) (1) {};
\node[main node] at (2,2) (2) {};
\node[main node] at (3,1) (3) {};
\node[main node] at (4,2) (4) {};
\node[main node] at (5,1) (5) {};
\node[main node] at (6,0) (6) {};
\draw[ultra thick] (0) -- (1) -- (2) -- (3)-- (4)-- (5)-- (6);

\end{tikzpicture}
\qquad
\begin{tikzpicture}[scale = \fig2scale, xscale=1,yscale=1, font=\scriptsize\sffamily, thick,main node/.style={circle,inner sep=0.5mm,draw, fill}]
\def\xlb{0}; \def\xub{6}; \def\ylb{0}; \def\yub{3}; \def\buf{0};
\foreach \x in {\xlb ,...,\xub}
{    \ifthenelse{\NOT 0 = \x}{\draw[thick](\x ,-2pt) -- (\x ,2pt);}{}
\draw[dotted, thick](\x,\ylb- \buf) -- (\x,\yub + \buf);}
\foreach \y in {\ylb ,...,\yub}
{    \ifthenelse{\NOT 0 = \y}{\draw[thick](-2pt, \y) -- (2pt, \y);}{}
\draw[dotted, thick](\xlb- \buf, \y) -- (\xub + \buf, \y);}
\node[main node] at (0,0) (0) {};
\node[main node] at (1,1) (1) {};
\node[main node] at (2,2) (2) {};
\node[main node] at (3,1) (3) {};
\node[main node] at (5,1) (4) {};
\node[main node] at (6,0) (5) {};
\draw[ultra thick] (0) -- (1) -- (2) -- (3)-- (4)-- (5);

\end{tikzpicture}
\qquad

\\
\\
\begin{tikzpicture}[scale = \fig2scale, xscale=1,yscale=1, font=\scriptsize\sffamily, thick,main node/.style={circle,inner sep=0.5mm,draw, fill}]
\def\xlb{0}; \def\xub{6}; \def\ylb{0}; \def\yub{3}; \def\buf{0};
\foreach \x in {\xlb ,...,\xub}
{    \ifthenelse{\NOT 0 = \x}{\draw[thick](\x ,-2pt) -- (\x ,2pt);}{}
\draw[dotted, thick](\x,\ylb- \buf) -- (\x,\yub + \buf);}
\foreach \y in {\ylb ,...,\yub}
{    \ifthenelse{\NOT 0 = \y}{\draw[thick](-2pt, \y) -- (2pt, \y);}{}
\draw[dotted, thick](\xlb- \buf, \y) -- (\xub + \buf, \y);}
\node[main node] at (0,0) (0) {};
\node[main node] at (1,1) (1) {};
\node[main node] at (2,2) (2) {};
\node[main node] at (3,1) (3) {};
\node[main node] at (4,0) (4) {};
\node[main node] at (5,1) (5) {};
\node[main node] at (6,0) (6) {};
\draw[ultra thick] (0) -- (1) -- (2) -- (3)-- (4)-- (5)-- (6);

\end{tikzpicture}
\qquad
\begin{tikzpicture}[scale = \fig2scale, xscale=1,yscale=1, font=\scriptsize\sffamily, thick,main node/.style={circle,inner sep=0.5mm,draw, fill}]
\def\xlb{0}; \def\xub{6}; \def\ylb{0}; \def\yub{3}; \def\buf{0};
\foreach \x in {\xlb ,...,\xub}
{    \ifthenelse{\NOT 0 = \x}{\draw[thick](\x ,-2pt) -- (\x ,2pt);}{}
\draw[dotted, thick](\x,\ylb- \buf) -- (\x,\yub + \buf);}
\foreach \y in {\ylb ,...,\yub}
{    \ifthenelse{\NOT 0 = \y}{\draw[thick](-2pt, \y) -- (2pt, \y);}{}
\draw[dotted, thick](\xlb- \buf, \y) -- (\xub + \buf, \y);}
\node[main node] at (0,0) (0) {};
\node[main node] at (1,1) (1) {};
\node[main node] at (3,1) (2) {};
\node[main node] at (4,2) (3) {};
\node[main node] at (5,1) (4) {};
\node[main node] at (6,0) (5) {};
\draw[ultra thick] (0) -- (1) -- (2) -- (3)-- (4)-- (5);

\end{tikzpicture}
\qquad
\begin{tikzpicture}[scale = \fig2scale, xscale=1,yscale=1, font=\scriptsize\sffamily, thick,main node/.style={circle,inner sep=0.5mm,draw, fill}]
\def\xlb{0}; \def\xub{6}; \def\ylb{0}; \def\yub{3}; \def\buf{0};
\foreach \x in {\xlb ,...,\xub}
{    \ifthenelse{\NOT 0 = \x}{\draw[thick](\x ,-2pt) -- (\x ,2pt);}{}
\draw[dotted, thick](\x,\ylb- \buf) -- (\x,\yub + \buf);}
\foreach \y in {\ylb ,...,\yub}
{    \ifthenelse{\NOT 0 = \y}{\draw[thick](-2pt, \y) -- (2pt, \y);}{}
\draw[dotted, thick](\xlb- \buf, \y) -- (\xub + \buf, \y);}
\node[main node] at (0,0) (0) {};
\node[main node] at (1,1) (1) {};
\node[main node] at (3,1) (2) {};
\node[main node] at (5,1) (3) {};
\node[main node] at (6,0) (4) {};
\draw[ultra thick] (0) -- (1) -- (2) -- (3)-- (4);

\end{tikzpicture}
\qquad
\begin{tikzpicture}[scale = \fig2scale, xscale=1,yscale=1, font=\scriptsize\sffamily, thick,main node/.style={circle,inner sep=0.5mm,draw, fill}]
\def\xlb{0}; \def\xub{6}; \def\ylb{0}; \def\yub{3}; \def\buf{0};
\foreach \x in {\xlb ,...,\xub}
{    \ifthenelse{\NOT 0 = \x}{\draw[thick](\x ,-2pt) -- (\x ,2pt);}{}
\draw[dotted, thick](\x,\ylb- \buf) -- (\x,\yub + \buf);}
\foreach \y in {\ylb ,...,\yub}
{    \ifthenelse{\NOT 0 = \y}{\draw[thick](-2pt, \y) -- (2pt, \y);}{}
\draw[dotted, thick](\xlb- \buf, \y) -- (\xub + \buf, \y);}
\node[main node] at (0,0) (0) {};
\node[main node] at (1,1) (1) {};
\node[main node] at (3,1) (2) {};
\node[main node] at (4,0) (3) {};
\node[main node] at (5,1) (4) {};
\node[main node] at (6,0) (5) {};
\draw[ultra thick] (0) -- (1) -- (2) -- (3)-- (4)-- (5);

\end{tikzpicture}
\\
\\
\begin{tikzpicture}[scale = \fig2scale, xscale=1,yscale=1, font=\scriptsize\sffamily, thick,main node/.style={circle,inner sep=0.5mm,draw, fill}]
\def\xlb{0}; \def\xub{6}; \def\ylb{0}; \def\yub{3}; \def\buf{0};
\foreach \x in {\xlb ,...,\xub}
{    \ifthenelse{\NOT 0 = \x}{\draw[thick](\x ,-2pt) -- (\x ,2pt);}{}
\draw[dotted, thick](\x,\ylb- \buf) -- (\x,\yub + \buf);}
\foreach \y in {\ylb ,...,\yub}
{    \ifthenelse{\NOT 0 = \y}{\draw[thick](-2pt, \y) -- (2pt, \y);}{}
\draw[dotted, thick](\xlb- \buf, \y) -- (\xub + \buf, \y);}
\node[main node] at (0,0) (0) {};
\node[main node] at (1,1) (1) {};
\node[main node] at (2,0) (2) {};
\node[main node] at (3,1) (3) {};
\node[main node] at (4,2) (4) {};
\node[main node] at (5,1) (5) {};
\node[main node] at (6,0) (6) {};
\draw[ultra thick] (0) -- (1) -- (2) -- (3)-- (4)-- (5)-- (6);

\end{tikzpicture}
\qquad
\begin{tikzpicture}[scale = \fig2scale, xscale=1,yscale=1, font=\scriptsize\sffamily, thick,main node/.style={circle,inner sep=0.5mm,draw, fill}]
\def\xlb{0}; \def\xub{6}; \def\ylb{0}; \def\yub{3}; \def\buf{0};
\foreach \x in {\xlb ,...,\xub}
{    \ifthenelse{\NOT 0 = \x}{\draw[thick](\x ,-2pt) -- (\x ,2pt);}{}
\draw[dotted, thick](\x,\ylb- \buf) -- (\x,\yub + \buf);}
\foreach \y in {\ylb ,...,\yub}
{    \ifthenelse{\NOT 0 = \y}{\draw[thick](-2pt, \y) -- (2pt, \y);}{}
\draw[dotted, thick](\xlb- \buf, \y) -- (\xub + \buf, \y);}
\node[main node] at (0,0) (0) {};
\node[main node] at (1,1) (1) {};
\node[main node] at (2,0) (2) {};
\node[main node] at (3,1) (3) {};
\node[main node] at (5,1) (4) {};
\node[main node] at (6,0) (5) {};
\draw[ultra thick] (0) -- (1) -- (2) -- (3)-- (4)-- (5);

\end{tikzpicture}
\qquad
\begin{tikzpicture}[scale = \fig2scale, xscale=1,yscale=1, font=\scriptsize\sffamily, thick,main node/.style={circle,inner sep=0.5mm,draw, fill}]
\def\xlb{0}; \def\xub{6}; \def\ylb{0}; \def\yub{3}; \def\buf{0};
\foreach \x in {\xlb ,...,\xub}
{    \ifthenelse{\NOT 0 = \x}{\draw[thick](\x ,-2pt) -- (\x ,2pt);}{}
\draw[dotted, thick](\x,\ylb- \buf) -- (\x,\yub + \buf);}
\foreach \y in {\ylb ,...,\yub}
{    \ifthenelse{\NOT 0 = \y}{\draw[thick](-2pt, \y) -- (2pt, \y);}{}
\draw[dotted, thick](\xlb- \buf, \y) -- (\xub + \buf, \y);}
\node[main node] at (0,0) (0) {};
\node[main node] at (1,1) (1) {};
\node[main node] at (2,0) (2) {};
\node[main node] at (3,1) (3) {};
\node[main node] at (4,0) (4) {};
\node[main node] at (5,1) (5) {};
\node[main node] at (6,0) (6) {};
\draw[ultra thick] (0) -- (1) -- (2) -- (3)-- (4)-- (5)-- (6);

\end{tikzpicture}
\qquad

\end{array}
\]

\caption{The $s(4) = 11$ small Schr\"oder paths from $(0,0)$ to $(6,0)$}\label{fig2}
\end{figure}
\end{center}
\end{itemize}

To see that $|\T_n| = |\P_n|$ for every $n \geq 1$, we describe the well-known ``walk around the tree'' procedure that maps plane rooted trees to lattice paths.

\begin{definition}\label{defnTreeToPath}
Given a Schr\"oder tree $T \in \T_n$, construct the path $\Psi(T) \in \P_n$ as follows:

\begin{enumerate}
\item
Suppose $T$ has $q$ nodes. We perform a preorder traversal of $T$, and label the nodes $a_1, a_2, \ldots, a_q$ in that order.
\item
Notice that $a_1$ must be the root of $T$. For every $i \geq 2$, define the function
\[
\psi(a_i) = 
\begin{cases}
U & \tn{if $a_i$ is the leftmost child of its parent;}\\
D & \tn{if $a_i$ is the rightmost child of its parent;}\\
F & \tn{otherwise.}
\end{cases}
\]
\item
Define
\[
\Psi(T) = \psi(a_2) \psi(a_3) \cdots \psi(a_q).
\]
\end{enumerate}
\end{definition}

Figure~\ref{fig3} illustrates the mapping $\Psi$ for a particular tree. It is not hard to check that $\Psi : \T_n \to \P_n$ is indeed a bijection.

\begin{center}
\begin{figure}[h]
\[
\begin{array}{ccc}

\begin{tikzpicture}
[scale=0.6,xscale=1.3, thick,main node/.style={circle,inner sep=0.3mm,draw,font=\scriptsize\sffamily}]
  \node[main node] at (3.5,6) (0) {$a_1$};
    \node[main node, label={[label distance=-0.1cm]180:$U$}] at (2,4) (l) {$a_2$};
    \node[main node, label={[label distance=-0.1cm]180:$U$}] at (0.5,2) (l1){$a_3$};
    \node[main node, label={[label distance=-0.1cm]180:$F$}] at (2,2) (l2){$a_4$};
    \node[main node, label={[label distance=-0.1cm]180:$U$}] at (1,0) (l21){$a_5$};
    \node[main node, label={[label distance=-0.1cm]180:$D$}] at (3,0) (l22){$a_6$};
    \node[main node, label={[label distance=-0.1cm]180:$D$}] at (3.5,2) (l3){$a_7$};
    \node[main node, label={[label distance=-0.1cm]180:$D$}] at (5,4) (r) {$a_8$};

  \path[every node/.style={font=\sffamily}]
    (0) edge (l)
    (0) edge (r)
                (l) edge (l1)
                 (l) edge (l2)
            (l) edge (l3)
(l2) edge (l21)
(l2) edge (l22);
\end{tikzpicture}

&
\raisebox{2cm}{$\longrightarrow$}
&
\raisebox{1.25cm}{
\begin{tikzpicture}[scale = 0.5, xscale=1,yscale=1, font=\scriptsize\sffamily, thick,main node/.style={circle,inner sep=0.5mm,draw, fill}]
\def\xlb{0}; \def\xub{8}; \def\ylb{0}; \def\yub{3}; \def\buf{0};
\foreach \x in {\xlb ,...,\xub}
{    \ifthenelse{\NOT 0 = \x}{\draw[thick](\x ,-2pt) -- (\x ,2pt);}{}
\draw[dotted, thick](\x,\ylb- \buf) -- (\x,\yub + \buf);}
\foreach \y in {\ylb ,...,\yub}
{    \ifthenelse{\NOT 0 = \y}{\draw[thick](-2pt, \y) -- (2pt, \y);}{}
\draw[dotted, thick](\xlb- \buf, \y) -- (\xub + \buf, \y);}

\node[main node] at (0,0) (0) {};
\node[main node] at (1,1) (1) {};
\node[main node] at (2,2) (2) {};
\node[main node] at (4,2) (3) {};
\node[main node] at (5,3) (4) {};
\node[main node] at (6,2) (5) {};
\node[main node] at (7,1) (6) {};
\node[main node] at (8,0) (7) {};
\draw[ultra thick] (0) -- (1) -- (2) -- (3)-- (4)-- (5)-- (6)-- (7);

\end{tikzpicture}
}
\\
\\
T \in \T_5 & & \Psi(T) \in \P_5\\
\end{array}
\]
\caption{Illustrating the tree-to-path mapping $\Psi$ (Definition~\ref{defnTreeToPath})}\label{fig3}
\end{figure}
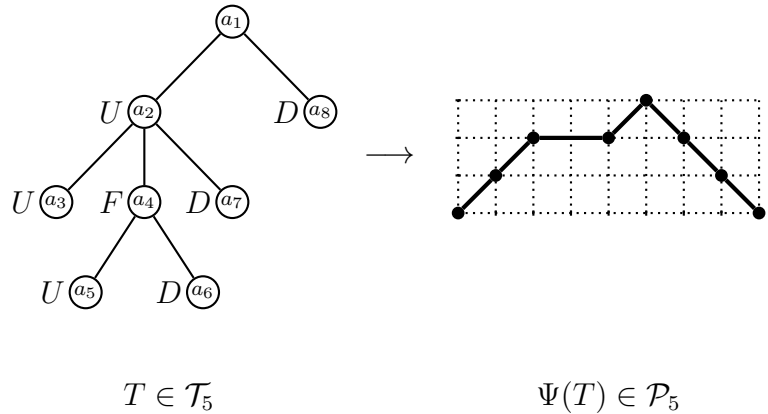
\end{center}

For more historical context, properties, and combinatorial interpretations of the small Schr\"oder numbers, the reader may refer to~\cite{Stanley97, Stanley99, ShapiroS00, Gessel09}.

\subsection{Weighted small Schr\"oder numbers}

Before we describe the weighted small Schr\"oder numbers, let us take a closer look at the Schr\"oder trees $\T_n$. Given integers $n \geq 1$ and $0 \leq k < n$, let $\T_{n,k} \subseteq \T_n$ be the set of Schr\"oder trees with $n$ leaves and $k$ internal nodes. Furthermore, let $s(n,k) = |T_{n,k}|$. For instance, Figure~\ref{fig1} shows that $s(4,0) = 0$, $s(4,1) = 1$, and $s(4,2) = s(4,3) = 5$. More generally, the numbers $s(n,k)$ produce the following triangle (\seqnum{A086810} in the OEIS).

\begin{center}
\begin{tabular}{r|rrrrrrr}
$s(n,k)$ & $k= 0$ & 1 & 2 & 3 & 4 & 5 & $\cdots$ \\
\hline
$n=1$  & 1 & & & & & &\\
2  &0 &1 & & & & &\\
3  &0 & 1& 2& & & &\\
4  &0 &1 & 5& 5& && \\
5  &0 & 1&9 &21 &14 && \\
6  &0  &1 &14 &56 &84 &42 &\\
$\vdots$  & $\vdots$ &$\vdots$ &$\vdots$ &$\vdots$ &$\vdots$ &$\vdots$ &$ \ddots$ \\
\end{tabular}
\end{center}
In general,
\[
s(n,k) = 
\begin{cases}
1 & \tn{if $n=1$ and $k=0$;}\\
\frac{1}{n-1} \binom{n-1}{k} \binom{n+k-1}{n} & \tn{if $n \geq 2$.}
\end{cases}
\]
(See, for instance,~\cite[Lemma 1]{GeffnerN17} for a proof.) Notice that $s(n) = \sum_{k=0}^{n-1} s(n,k)$. Next, we likewise define $\P_{n,k} \subseteq \P_n$ to be the set of small Schr\"oder paths in $\P_n$ with exactly $k$ up steps. Now notice that given $T \in \T_{n,k}$, there must be exactly $k$ nodes in $T$ that is each the leftmost child of its parent. Thus, $\Psi(T)$ would have exactly $k$ up steps, and so $\Psi$ is in fact a bijection between $\T_{n,k}$ and $\P_{n,k}$ for every $n$ and $k$, and it follows that $s(n,k) = |\P_{n,k}|$. Moreover, observe that $s(n,n-1)$ counts the number of Schr\"oder paths from $(0,0)$ to $(2n-2,0)$ with $n-1$ up steps (and hence $n-1$ down steps). These paths must then have no flat steps, and therefore are in fact Dyck paths. Hence, $s(n,n-1)$ gives the $n^{\tn{th}}$ Catalan number~(\seqnum{A000108} in the OEIS).

We are now ready to define the sequences that are of our main focus in this manuscript. Given a real number $d$, we define
\begin{equation}\label{eqSdn}
s_d(n) = \sum_{k=0}^{n-1} s(n,k) d^k
\end{equation}
for every integer $n \geq 1$. Intuitively, one can interpret assigning a weight of $d^k$ to each tree in $\T_{n,k}$, and let $s_d(n)$ be the total weight of all trees in $\T_n$.

Clearly, $s_1(n) = s(n)$, so the above is indeed a generalization of the small Schr\"oder numbers. For $d \geq 2$, we get the following sequences:

\begin{center}
\begin{tabular}{l|rrrrrrrrr}
$n$ & 1 & 2 & 3 & 4 & 5 & 6 & 7 & 8 & $\cdots$ \\
\hline
$s_2(n)$ & 1 & 2 & 10 & 62 & 430 & 3194 & 24850 & 199910 & $\cdots$ \\

$s_3(n)$ & 1 & 3 & 21 & 183 & 1785 & 18651 & 204141 & 2310447 & $\cdots$ \\

$s_4(n)$ & 1 & 4 & 36 &404 & 5076 & 68324 & 963396 & 14046964 & $\cdots$ \\
\end{tabular}
\end{center}

(See  \seqnum{A107841}, \seqnum{A131763}, and  \seqnum{A131765} for the cases $d=2,3,4$ respectively.)
While we could not find any literature in which these sequences were the main focus of study, they (especially $s_2(n)$) have made appearances in many areas, such as queuing theory~\cite{AbateW10}, embedded Riordan arrays~\cite{Barry14}, operads from posets~\cite{Giraudo16}, and vorticity equations in fluid dynamics~\cite{Lam13}. Chan and Pan~\cite{ChenP17} also came across these sequences when enumerating a certain family of valley-type weighted Dyck paths, and related these quantities to a variant of generalized large Schr\"oder numbers that is somewhat similar to $s_d(n)$. We will make this connection more explicit in Section~\ref{sec13}. For other generalizations of small Schr\"oder numbers, see (among others)~\cite{Sulanke04, Schroder07, HuhP15}.

\subsection{A roadmap of this paper}

In Section~\ref{sec12a}, we will discuss some basic properties of $s_d(n)$ by studying its generating function. We will see that the tools for establishing known formulas for $s(n)$ readily extend to proving analogous results for $s_d(n)$. After that, we prove a few identities for $s_d(n)$ (Section 3), and describe several families of Dyck paths that are counted by $s_d(n)$ (Section 4). Our analysis of the weighted Schr\"oder numbers $s_d(n)$ leads to several results regarding Schr\"oder paths and Dyck paths, such as:
\begin{itemize}
\item
For every $n \geq 1$, the number of small Schr\"oder paths in $\P_n$ with an odd number of up steps and that with an even number of up steps differ by exactly one (Proposition~\ref{s_MinusOne});
\item
For every $n \geq 1$, $s_{k\l-1}(n)$ gives the number of Dyck paths from $(0,0)$ to $(2n-2,0)$ with $k$ possible colors for each up step ($U_1, \ldots, U_k$), $\ell$ possible colors for each down step ($D_1, \ldots, D_{\ell})$, and avoid peaks of type $U_1D_1$ (Proposition~\ref{s_dScambler}).
\end{itemize}
We close in Section~\ref{sec14} by mentioning some possible future research directions.

\section{Basic properties and formulas}\label{sec12a}

Let $d$ be a fixed real number, and consider the generating function $y = \sum_{n \geq 1} s_d(n)x^n$. We first establish a functional equation for $y$. 

\begin{proposition}\label{s_dGeneratingFunctionProp}
The generating function $y = \sum_{n \geq 1} s_d(n)x^n$ satisfies the functional equation
\begin{equation}\label{s_dGeneratingFunction}
(d+1)y^2 - (x+1)y + x = 0.
\end{equation}
\end{proposition}

\begin{proof}
Let $\T = \bigcup_{n \geq 1} \T_n$ be the set of all Schr\"oder trees, and for each tree $T \in \T$ we let $k(T), n(T)$ be the number of internal nodes and leaves of $T$, respectively. Then notice that for each $T \in \T$ where $n(T) \geq 2$, the root of $T$ must be an internal node with $\ell \geq 2$ subtrees $T_1, \ldots, T_{\ell}$, in which case $ \sum_{j=1}^{\ell} k(T_j) +1 = k(T)$ and  $ \sum_{j=1}^{\ell} n(T_j) = n(T)$.  Thus,
\begin{align*}
y &= \sum_{n \geq 1} s_d(n)x^n \\
&= \sum_{T \in \T} d^{k(T)} x^{n(T)}\\
&= \sum_{T \in \T_1}d^{k(T)} x^{n(T)} + sum_{T \in \T \setminus \T_1} d^{k(T)}x^{n(T)}\\
&= x + \sum_{\ell \geq 2} \sum_{T_1, \ldots, T_{\ell} \in \T} d^{1 + \sum_{j=1}^{\ell} k(T_j)} x^{\sum_{j=1}^{\ell} n(T_j)} \\
&= x + d \sum_{\ell \geq 2} y^{\ell}\\
&=  x + \frac{dy^2}{1-y},
\end{align*}
which easily rearranges to give~\eqref{s_dGeneratingFunction}.
\end{proof}

When $d \neq -1$, solving for $y$ in~\eqref{s_dGeneratingFunction} (and noting $[x^0]y = 0$) yields
\begin{equation}\label{s_dGeneratingFunction2}
y = \frac{1}{2(d+1)} \left(1+x - \sqrt{1-(4d+2)x+x^2} \right).
\end{equation}
We next use~\eqref{s_dGeneratingFunction} and~\eqref{s_dGeneratingFunction2} to prove a recurrence relation for $s_d(n)$ that would allow for very efficient computation of these numbers. For the case $d=1$, Stanley~\cite{Stanley97} proved the recurrence
\begin{equation}\label{s_1Recur}
ns(n) = 3(2n-3)s(n-1) - (n-3)s(n-2), \quad n \geq 3
\end{equation}
using generating functions. Subsequently, a combinatorial proof using weighted binary trees is given~\cite{FoataZ98}. Shortly after, Sulanke~\cite{Sulanke98} also showed that the above recurrence applies for the closely-related large Schr{\"o}der numbers using Schr{\"o}der paths.

Herein, we adapt Stanley's~\cite{Stanley97} argument to obtain a similar recurrence relation for $s_d(n)$ for all $d$.

\begin{proposition}\label{s_dRecur}
For every real number $d$, $s_d(1) =1, s_d(2) = d$, and
\begin{equation}\label{s_dRecur1}
ns_d(n) = (2d+1) (2n-3)s_d(n-1) - (n-3) s_d(n-2)
\end{equation}
for all $n \geq 3$.
\end{proposition}

\begin{proof}
First, $s_d(1) = 1$ and $s_d(2) = d$ follows readily from the definition of $s_d(n)$. For the recurrence, we start with~\eqref{s_dGeneratingFunction} and differentiate both sides with respect to $x$ to obtain
\[
2(d+1)y y' - y - (x+1)y' + 1 =0.
\]
Rearranging gives 
\[
y' = \frac{y-1}{ 2(d+1)y - x - 1}.
\]
We also have 
\[
2(d+1)y - x - 1 = -\sqrt{x^2 - (4d+2)x+1}.
\]
This obviously holds when $d=-1$, and can be derived from~\eqref{s_dGeneratingFunction2} for all other values of $d$. Thus,
\begin{align*}
y' &= \frac{y-1}{ 2(d+1)y - x - 1} =  \frac{y-1}{-\sqrt{x^2-(4d+2)x+1}}= \frac{(y-1)(2(d+1)y-x-1)}{x^2-(4d+2)x+1}\\
&=  \frac{2(d+1)y^2 - (2d+x+3)y+x+1}{x^2-(4d+2)x+1} =  \frac{(x-2d-1)y-x+1}{x^2-(4d+2)x+1},
\end{align*}
where the last equality made use of the fact that $(d+1)y^2 = (x+1)y-x$ from~\eqref{s_dGeneratingFunction}. From the above, we obtain the equation
\begin{equation}\label{s_dRecur2}
(x^2-(4d+2)x+1)y' - (x-2d-1)y - x + 1 = 0.
\end{equation}
Since $y = \sum_{n \geq 0} s_d(n)x^n$ and $y' = \sum_{n \geq 0} (n+1) s_d(n+1) x^n$, taking the coefficient of $x^{n-1}$ of \eqref{s_dRecur2} (for any $n \geq 3$) yields
\[
(n-2) s_d(n-2)  - (4d+2)(n-1)s_d(n-1) + ns_d(n) - s_d(n-2) + (2d+1)s_d(n-1) = 0,
\]
which can be rearranged to give the desired recurrence~\eqref{s_dRecur1}.
\end{proof}

Next, we look into the asymptotic behavior of $s_d(n)$. The asymptotic formula for $s_1(n)$ is well known --- see, for instance,~\cite[p.\ 474]{FlajoletS09}. Here, we extend that formula to one that applies for arbitrary, positive $d$.

\begin{proposition}\label{s_dAsymptotic}
For every real number $d > 0$, as $n \to \infty$,
\[
s_d(n) \sim \left( \frac{ (\sqrt{d+1}-\sqrt{d})\cdot d^{1/4}}{ 2 (d+1)^{3/4} \pi^{1/2}} \right) \cdot n^{-3/2} \cdot \left(2d+1+2\sqrt{d^2+d}\right)^n.
\]
\end{proposition}

\begin{proof}
We follow the template given in~\cite[Theorem VI.6, p.\ 420]{FlajoletS09} to prove our claim. First, recall that $y = \sum_{n \geq 1} s_d(n) x^n$ satisfies the functional equation $y = x \phi(y)$ where $\phi(y) = \frac{1-y}{1-(d+1)y}$. We need to verify the following analytic conditions for $\phi$:
\begin{itemize}
\item[$H_1$:]
$\phi$ is a nonlinear function that is analytic at $0$ with $\phi(0) \neq 0$ and $[z^n] \phi(z) \geq 0$ for all $n \geq 0$.
\item[$H_2$:]
Within the open disc of convergence of $\phi$ at $0$, $|z| < R$, there exists a (then necessarily unique) positive solution $s$ to the characteristic equation $\phi(s) = s\phi'(s)$.
\end{itemize}

Notice that, expanding $\phi(z)$, we obtain
\[
\phi(z) = \frac{1-z}{1-(d+1)z} = 1 + \sum_{n \geq 1} \left((d+1)^n - (d+1)^{n-1} \right)z^n = 1 + d \sum_{n \geq 1} (d+1)^{n-1}z^n,
\]
and so $H_1$ holds for all $d >0$. For $H_2$, the radius of convergence of $\phi(z)$ (which is a geometric series) is obviously $R= \frac{1}{d+1}$. Now since $\phi'(z) = \frac{d}{(1 + (d+1)z)^2}$, solving $\phi(s) = s\phi'(s)$ yields one solution $s = 1 - \sqrt{ \frac{d}{d+1}}$, which lies in $(0, R)$ given $d>0$. Thus, $H_2$ holds as well.

Hence, the aforementioned result in~\cite{FlajoletS09} applies, and
\begin{equation}\label{s_dasymp1}
s_d(n) \sim \sqrt{\frac{\phi(s)}{2\phi''(s)}} \cdot \frac{ \rho^{n}}{\sqrt{\pi n^3}},
\end{equation}
where $\rho = \frac{\phi(s)}{s}$. It is not hard to check that
\begin{align*}
\phi(s) &= \frac{1}{\sqrt{d+1}(\sqrt{d+1} - \sqrt{d})},\\
 \phi''(s)& = \frac{4(d+1)}{\sqrt{d} \left( \sqrt{d+1} - \sqrt{d} \right)^3},\\
  \rho &= 2d+1 + 2\sqrt{d^2+d}.
  \end{align*}
Substituting these expressions into~\eqref{s_dasymp1} and simplifying gives the desired result.
\end{proof}

In particular, from Proposition~\ref{s_dAsymptotic} we obtain that the sequence $s_d(n)$ has growth rate
\[
\lim_{n \to \infty} \frac{s_d(n+1)}{s_d(n)} = \rho = 2d+1 + 2\sqrt{d^2+d} \in (4d+1, 4d+2)
\]
for all $d > 0$. 

\section{Some identities and implications}\label{sec12b}

In this section, we will prove several identities related to $s_d(n)$, and describe their combinatorial implications. We will first focus on two special sequences $s_{-1/2}(n)$ and $s_{-1}(n)$, then prove an identity that seems to have a natural connection with large Schr\"oder numbers.

\subsection{The case $d= -1/2$}

Notice that in the recurrence~\eqref{s_dRecur1}, the coefficient of $s_d(n-1)$ vanishes when $d = \frac{-1}{2}$. In this case, we obtain the sequence:

\begin{center}
\begin{tabular}{l|rrrrrrrrrrrrr}
$n$ & 1 & 2 & 3 & 4 & 5 & 6 & 7 & 8 & 9 & 10 & 11 & 12 & $\cdots$ \\
\hline
$s_{-1/2}(n)$ & 1 & $\frac{-1}{2}$ & 0 & $\frac{1}{4}$ & 0 & $\frac{-2}{8}$ & 0 & $\frac{5}{16}$ & 0 & $\frac{-14}{32}$ & 0 & $\frac{42}{64}$ & $\cdots$ \\
\end{tabular}
\end{center}

For every $n \geq 1$, let $c(n) = \frac{1}{n} \binom{2n-2}{n-1}$ denote the $n^{\tn{th}}$ Catalan number. It is easy to check that $c(1) = 1$, and that
\begin{equation}\label{c_nRecur}
c(n) = \frac{2 (2n-3)}{n} c(n-1)
\end{equation}
for all $n \geq 2$. Using this recurrence of $c(n)$ and Proposition~\ref{s_dRecur}, we prove the following:

\begin{proposition}\label{s_MinusOneHalf}
For every integer $m \geq 1$, $s_{-1/2}(2m+1) = 0$ and $s_{-1/2}(2m) = \frac{(-1)^m}{2^{2m-1}} c(m)$.
\end{proposition}

\begin{proof}
We prove our claim by induction on $m$. First, in the case of $d= \frac{-1}{2}$,~\eqref{s_dRecur1} reduces to
\begin{equation}\label{s_MinusOneHalf1}
s_{-1/2}(n) = \frac{-(n-3)}{n} s_{-1/2}(n-2).
\end{equation}
Thus, $s_{-1/2}(3) = 0 s_{-1/2}(1) = 0$. From there on, since $s_{-1/2}(2m+1)$ is a multiple of $s_{-1/2}(2m-1)$ for all $m \geq 1$, we obtain that $s_{-1/2}(2m+1) = 0$ for all $m \geq 1$.

Next, we establish the claim for $s_{-1/2}(2m)$. When $m=1$, 
\[
s_{-1/2}(2) = \frac{-1}{2} = \frac{(-1)^1}{2^{2-1}} c(1),
\]
so the base case holds. For the inductive step, notice that
\begin{align*}
s_{-1/2}(2m) 
&= \frac{-(2m-3)}{2m} s_{-1/2}(2m-2) \\
&= \frac{-(2m-3)}{2m} \left( \frac{(-1)^{n-1}}{2^{2m-3}} c(m-1) \right)\\
&= \frac{(-1)^{n}}{2^{2m-1}} \left( \frac{2(2m-3)}{m}  c(m-1)  \right)\\
&= \frac{(-1)^{n}}{2^{2m-1}} c(m),
\end{align*}
where the last equality follows from~\eqref{c_nRecur}. This finishes our proof.
\end{proof}

We will return to the quantity $s_{-1/2}(n)$ subsequently in Section~\ref{sec13}.

\subsection{The case $d=-1$}

Next, we turn to the case when $d=-1$. In this case, \eqref{s_dGeneratingFunction} gives
\[
y=  \frac{x}{1-x} = \sum_{n \geq 1} (-1)^{n-1}x^i,
\]
and so $s_{-1}(n) = (-1)^{n-1}$ for every $n \geq 1$. Since $s_{-1}(n) = \sum_{k \geq 0} s(n,k) (-1)^k$ by definition, we have shown the following:

\begin{proposition}\label{s_MinusOne}
For every integer $n \geq 1$, 
\[
\sum_{k~\tn{odd}} s(n,k) = \sum_{k~\tn{even}} s(n,k)+(-1)^{n}.
\]
\end{proposition}

Proposition~\ref{s_MinusOne} implies that, for every $n$, the number of trees in $\T_n$ with an odd number of internal nodes and that with an even number of internal nodes differ by exactly $1$. Likewise, in terms of small Schr\"oder paths, we now know that the number of paths in $\P_n$ with an odd number of up steps is exactly $1$ away from that with an even number of up steps. This leads to the following consequence, which might be folklore but we could not find a mention of in the literature:

\begin{corollary}\label{s_nOdd}
The small Schr\"oder numbers $s(n)$ is odd for all $n \geq 1$.
\end{corollary}

Corollary~\ref{s_nOdd} further implies that, for all odd, positive integers $d$ and for all $n \geq 1$, $s_d(n)$ is odd (since then $d^k$ is odd for all $k \geq 0$, and so $s_d(n)$ would be a sum of an odd number of odd quantities). Likewise, we obtain that for all $n \geq 2$, the $n^{\tn{th}}$ large Schr\"oder number (which is twice the $n^{\tn{th}}$ small Schr\"oder number) is congruent to $2$ (mod $4$).

While we have already described a very short algebraic proof to Proposition~\ref{s_MinusOne}, we will also provide a simple combinatorial proof.

\begin{proof}[Proof of Proposition~\ref{s_MinusOne}]
Given a fixed $n$, consider the set of paths $\P_n$. Let $Q = U^{n-1} D^{n-1}$ (i.e., the path consisting of $n-1$ up steps followed by $n-1$ down steps), and define
\[
\P^O_n = \left( \bigcup_{k~\tn{odd}} \P_{n,k} \right) \setminus \set{Q}, \quad 
\P^E_n = \left( \bigcup_{k~\tn{even}} \P_{n,k} \right) \setminus \set{Q}.
\]
To prove our claim, it then suffices to find a bijection between $\P^O_n$ and $\P^E_n$ for every $n$. Given a path $P \in \P_n$ where $P \neq Q$, $P$ must then contain either a flat step or a valley (or both). We
define $\alpha(P)$ as follows:
\begin{itemize}
\item[(1)]
If the first valley in $P$ occurs before the first flat step, we write $P= P_1 DU P_2$ where $P_1$ does not contain a flat step, and define $\alpha(P) = P_1 F P_2$.
\item[(2)]
Otherwise, the first flat step occurs before the first valley. In this case, we write $P$ as $P_1 F P_2$ where $P_1$ does not contain an instance of $DU$, and define $\alpha(P) = P_1DU P_2$.
\end{itemize}
Figure~\ref{figS_MinusOne} illustrates the mapping $\alpha$ on paths in $\P_4$. Intuitively, given a path, $\alpha$ finds the first appearance of a valley $DU$ or a flat step $F$, whichever comes first. If it is a valley, then $\alpha$ replaces it by a flat step, and vice versa. 

\def\figS_MinusOneScale{0.3}
\begin{center}
\begin{figure}[h]

\[
\begin{array}{cccccc}
\begin{tikzpicture}[scale = \figS_MinusOneScale, xscale=1,yscale=1, font=\scriptsize\sffamily, thick,main node/.style={circle,inner sep=0.5mm,draw, fill}]
\def\xlb{0}; \def\xub{6}; \def\ylb{0}; \def\yub{3}; \def\buf{0};
\foreach \x in {\xlb ,...,\xub}
{    \ifthenelse{\NOT 0 = \x}{\draw[thick](\x ,-2pt) -- (\x ,2pt);}{}
\draw[dotted, thick](\x,\ylb- \buf) -- (\x,\yub + \buf);}
\foreach \y in {\ylb ,...,\yub}
{    \ifthenelse{\NOT 0 = \y}{\draw[thick](-2pt, \y) -- (2pt, \y);}{}
\draw[dotted, thick](\xlb- \buf, \y) -- (\xub + \buf, \y);}
\node[main node] at (0,0) (0) {};
\node[main node] at (1,1) (1) {};
\node[main node] at (2,2) (2) {};
\node[main node] at (3,1) (3) {};
\node[main node] at (4,2) (4) {};
\node[main node] at (5,1) (5) {};
\node[main node] at (6,0) (6) {};
\draw[ultra thick] (0) -- (1) -- (2) -- (3)-- (4)-- (5)-- (6);

\end{tikzpicture}

&
\begin{tikzpicture}[scale = \figS_MinusOneScale, xscale=1,yscale=1, font=\scriptsize\sffamily, thick,main node/.style={circle,inner sep=0.5mm,draw, fill}]
\def\xlb{0}; \def\xub{6}; \def\ylb{0}; \def\yub{3}; \def\buf{0};
\foreach \x in {\xlb ,...,\xub}
{    \ifthenelse{\NOT 0 = \x}{\draw[thick](\x ,-2pt) -- (\x ,2pt);}{}
\draw[dotted, thick](\x,\ylb- \buf) -- (\x,\yub + \buf);}
\foreach \y in {\ylb ,...,\yub}
{    \ifthenelse{\NOT 0 = \y}{\draw[thick](-2pt, \y) -- (2pt, \y);}{}
\draw[dotted, thick](\xlb- \buf, \y) -- (\xub + \buf, \y);}
\node[main node] at (0,0) (0) {};
\node[main node] at (1,1) (1) {};
\node[main node] at (2,2) (2) {};
\node[main node] at (3,1) (3) {};
\node[main node] at (4,0) (4) {};
\node[main node] at (5,1) (5) {};
\node[main node] at (6,0) (6) {};
\draw[ultra thick] (0) -- (1) -- (2) -- (3)-- (4)-- (5)-- (6);

\end{tikzpicture}
&
\begin{tikzpicture}[scale = \figS_MinusOneScale, xscale=1,yscale=1, font=\scriptsize\sffamily, thick,main node/.style={circle,inner sep=0.5mm,draw, fill}]
\def\xlb{0}; \def\xub{6}; \def\ylb{0}; \def\yub{3}; \def\buf{0};
\foreach \x in {\xlb ,...,\xub}
{    \ifthenelse{\NOT 0 = \x}{\draw[thick](\x ,-2pt) -- (\x ,2pt);}{}
\draw[dotted, thick](\x,\ylb- \buf) -- (\x,\yub + \buf);}
\foreach \y in {\ylb ,...,\yub}
{    \ifthenelse{\NOT 0 = \y}{\draw[thick](-2pt, \y) -- (2pt, \y);}{}
\draw[dotted, thick](\xlb- \buf, \y) -- (\xub + \buf, \y);}
\node[main node] at (0,0) (0) {};
\node[main node] at (1,1) (1) {};
\node[main node] at (3,1) (2) {};
\node[main node] at (5,1) (3) {};
\node[main node] at (6,0) (4) {};
\draw[ultra thick] (0) -- (1) -- (2) -- (3)-- (4);

\end{tikzpicture}
&
\begin{tikzpicture}[scale = \figS_MinusOneScale, xscale=1,yscale=1, font=\scriptsize\sffamily, thick,main node/.style={circle,inner sep=0.5mm,draw, fill}]
\def\xlb{0}; \def\xub{6}; \def\ylb{0}; \def\yub{3}; \def\buf{0};
\foreach \x in {\xlb ,...,\xub}
{    \ifthenelse{\NOT 0 = \x}{\draw[thick](\x ,-2pt) -- (\x ,2pt);}{}
\draw[dotted, thick](\x,\ylb- \buf) -- (\x,\yub + \buf);}
\foreach \y in {\ylb ,...,\yub}
{    \ifthenelse{\NOT 0 = \y}{\draw[thick](-2pt, \y) -- (2pt, \y);}{}
\draw[dotted, thick](\xlb- \buf, \y) -- (\xub + \buf, \y);}
\node[main node] at (0,0) (0) {};
\node[main node] at (1,1) (1) {};
\node[main node] at (2,0) (2) {};
\node[main node] at (3,1) (3) {};
\node[main node] at (4,2) (4) {};
\node[main node] at (5,1) (5) {};
\node[main node] at (6,0) (6) {};
\draw[ultra thick] (0) -- (1) -- (2) -- (3)-- (4)-- (5)-- (6);

\end{tikzpicture}
&
\begin{tikzpicture}[scale = \figS_MinusOneScale, xscale=1,yscale=1, font=\scriptsize\sffamily, thick,main node/.style={circle,inner sep=0.5mm,draw, fill}]
\def\xlb{0}; \def\xub{6}; \def\ylb{0}; \def\yub{3}; \def\buf{0};
\foreach \x in {\xlb ,...,\xub}
{    \ifthenelse{\NOT 0 = \x}{\draw[thick](\x ,-2pt) -- (\x ,2pt);}{}
\draw[dotted, thick](\x,\ylb- \buf) -- (\x,\yub + \buf);}
\foreach \y in {\ylb ,...,\yub}
{    \ifthenelse{\NOT 0 = \y}{\draw[thick](-2pt, \y) -- (2pt, \y);}{}
\draw[dotted, thick](\xlb- \buf, \y) -- (\xub + \buf, \y);}
\node[main node] at (0,0) (0) {};
\node[main node] at (1,1) (1) {};
\node[main node] at (2,0) (2) {};
\node[main node] at (3,1) (3) {};
\node[main node] at (4,0) (4) {};
\node[main node] at (5,1) (5) {};
\node[main node] at (6,0) (6) {};
\draw[ultra thick] (0) -- (1) -- (2) -- (3)-- (4)-- (5)-- (6);

\end{tikzpicture}
&
\begin{tikzpicture}[scale = \figS_MinusOneScale, xscale=1,yscale=1, font=\scriptsize\sffamily, thick,main node/.style={circle,inner sep=0.5mm,draw, fill}]
\def\xlb{0}; \def\xub{6}; \def\ylb{0}; \def\yub{3}; \def\buf{0};
\foreach \x in {\xlb ,...,\xub}
{    \ifthenelse{\NOT 0 = \x}{\draw[thick](\x ,-2pt) -- (\x ,2pt);}{}
\draw[dotted, thick](\x,\ylb- \buf) -- (\x,\yub + \buf);}
\foreach \y in {\ylb ,...,\yub}
{    \ifthenelse{\NOT 0 = \y}{\draw[thick](-2pt, \y) -- (2pt, \y);}{}
\draw[dotted, thick](\xlb- \buf, \y) -- (\xub + \buf, \y);}

\node[main node] at (0,0) (0) {};
\node[main node] at (1,1) (1) {};
\node[main node] at (2,2) (2) {};
\node[main node] at (3,3) (3) {};
\node[main node] at (4,2) (4) {};
\node[main node] at (5,1) (5) {};
\node[main node] at (6,0) (6) {};
\draw[ultra thick] (0) -- (1) -- (2) -- (3)-- (4)-- (5)-- (6);

\end{tikzpicture}

\\
\updownarrow & \updownarrow & \updownarrow & \updownarrow & \updownarrow & \\

\begin{tikzpicture}[scale = \figS_MinusOneScale, xscale=1,yscale=1, font=\scriptsize\sffamily, thick,main node/.style={circle,inner sep=0.5mm,draw, fill}]
\def\xlb{0}; \def\xub{6}; \def\ylb{0}; \def\yub{3}; \def\buf{0};
\foreach \x in {\xlb ,...,\xub}
{    \ifthenelse{\NOT 0 = \x}{\draw[thick](\x ,-2pt) -- (\x ,2pt);}{}
\draw[dotted, thick](\x,\ylb- \buf) -- (\x,\yub + \buf);}
\foreach \y in {\ylb ,...,\yub}
{    \ifthenelse{\NOT 0 = \y}{\draw[thick](-2pt, \y) -- (2pt, \y);}{}
\draw[dotted, thick](\xlb- \buf, \y) -- (\xub + \buf, \y);}
\node[main node] at (0,0) (0) {};
\node[main node] at (1,1) (1) {};
\node[main node] at (2,2) (2) {};
\node[main node] at (4,2) (3) {};
\node[main node] at (5,1) (4) {};
\node[main node] at (6,0) (5) {};
\draw[ultra thick] (0) -- (1) -- (2) -- (3)-- (4)-- (5);

\end{tikzpicture}
&

\begin{tikzpicture}[scale = \figS_MinusOneScale, xscale=1,yscale=1, font=\scriptsize\sffamily, thick,main node/.style={circle,inner sep=0.5mm,draw, fill}]
\def\xlb{0}; \def\xub{6}; \def\ylb{0}; \def\yub{3}; \def\buf{0};
\foreach \x in {\xlb ,...,\xub}
{    \ifthenelse{\NOT 0 = \x}{\draw[thick](\x ,-2pt) -- (\x ,2pt);}{}
\draw[dotted, thick](\x,\ylb- \buf) -- (\x,\yub + \buf);}
\foreach \y in {\ylb ,...,\yub}
{    \ifthenelse{\NOT 0 = \y}{\draw[thick](-2pt, \y) -- (2pt, \y);}{}
\draw[dotted, thick](\xlb- \buf, \y) -- (\xub + \buf, \y);}
\node[main node] at (0,0) (0) {};
\node[main node] at (1,1) (1) {};
\node[main node] at (2,2) (2) {};
\node[main node] at (3,1) (3) {};
\node[main node] at (5,1) (4) {};
\node[main node] at (6,0) (5) {};
\draw[ultra thick] (0) -- (1) -- (2) -- (3)-- (4)-- (5);

\end{tikzpicture}
&
\begin{tikzpicture}[scale = \figS_MinusOneScale, xscale=1,yscale=1, font=\scriptsize\sffamily, thick,main node/.style={circle,inner sep=0.5mm,draw, fill}]
\def\xlb{0}; \def\xub{6}; \def\ylb{0}; \def\yub{3}; \def\buf{0};
\foreach \x in {\xlb ,...,\xub}
{    \ifthenelse{\NOT 0 = \x}{\draw[thick](\x ,-2pt) -- (\x ,2pt);}{}
\draw[dotted, thick](\x,\ylb- \buf) -- (\x,\yub + \buf);}
\foreach \y in {\ylb ,...,\yub}
{    \ifthenelse{\NOT 0 = \y}{\draw[thick](-2pt, \y) -- (2pt, \y);}{}
\draw[dotted, thick](\xlb- \buf, \y) -- (\xub + \buf, \y);}
\node[main node] at (0,0) (0) {};
\node[main node] at (1,1) (1) {};
\node[main node] at (2,0) (2) {};
\node[main node] at (3,1) (3) {};
\node[main node] at (5,1) (4) {};
\node[main node] at (6,0) (5) {};
\draw[ultra thick] (0) -- (1) -- (2) -- (3)-- (4)-- (5);

\end{tikzpicture}
&
\begin{tikzpicture}[scale = \figS_MinusOneScale, xscale=1,yscale=1, font=\scriptsize\sffamily, thick,main node/.style={circle,inner sep=0.5mm,draw, fill}]
\def\xlb{0}; \def\xub{6}; \def\ylb{0}; \def\yub{3}; \def\buf{0};
\foreach \x in {\xlb ,...,\xub}
{    \ifthenelse{\NOT 0 = \x}{\draw[thick](\x ,-2pt) -- (\x ,2pt);}{}
\draw[dotted, thick](\x,\ylb- \buf) -- (\x,\yub + \buf);}
\foreach \y in {\ylb ,...,\yub}
{    \ifthenelse{\NOT 0 = \y}{\draw[thick](-2pt, \y) -- (2pt, \y);}{}
\draw[dotted, thick](\xlb- \buf, \y) -- (\xub + \buf, \y);}
\node[main node] at (0,0) (0) {};
\node[main node] at (1,1) (1) {};
\node[main node] at (3,1) (2) {};
\node[main node] at (4,2) (3) {};
\node[main node] at (5,1) (4) {};
\node[main node] at (6,0) (5) {};
\draw[ultra thick] (0) -- (1) -- (2) -- (3)-- (4)-- (5);

\end{tikzpicture}
&
\begin{tikzpicture}[scale = \figS_MinusOneScale, xscale=1,yscale=1, font=\scriptsize\sffamily, thick,main node/.style={circle,inner sep=0.5mm,draw, fill}]
\def\xlb{0}; \def\xub{6}; \def\ylb{0}; \def\yub{3}; \def\buf{0};
\foreach \x in {\xlb ,...,\xub}
{    \ifthenelse{\NOT 0 = \x}{\draw[thick](\x ,-2pt) -- (\x ,2pt);}{}
\draw[dotted, thick](\x,\ylb- \buf) -- (\x,\yub + \buf);}
\foreach \y in {\ylb ,...,\yub}
{    \ifthenelse{\NOT 0 = \y}{\draw[thick](-2pt, \y) -- (2pt, \y);}{}
\draw[dotted, thick](\xlb- \buf, \y) -- (\xub + \buf, \y);}
\node[main node] at (0,0) (0) {};
\node[main node] at (1,1) (1) {};
\node[main node] at (3,1) (2) {};
\node[main node] at (4,0) (3) {};
\node[main node] at (5,1) (4) {};
\node[main node] at (6,0) (5) {};
\draw[ultra thick] (0) -- (1) -- (2) -- (3)-- (4)-- (5);

\end{tikzpicture}
&

\end{array}
\]

\caption{Illustrating the mapping $\alpha$ on paths in $\P_4$ (Proof of Proposition~\ref{s_MinusOne}). Paths on the first (resp., second) row have an odd (resp., even) number of up steps.}\label{figS_MinusOne}
\end{figure}
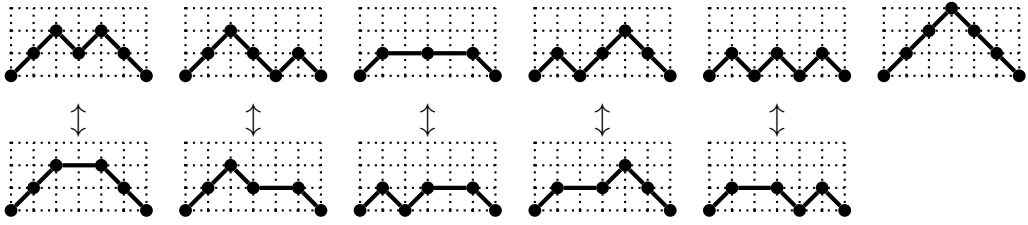
\end{center}

Now observe that in any case, the number of up steps in $P$ and in $\alpha(P)$ differ by exactly one, so $\alpha$ can be considered as a mapping from $\P^O_n$ to $\P^E_n$. Moreover, it is easy to see that $\alpha(\alpha(P)) = P$ for all paths $P \in \P_n \setminus \set{Q}$, and so $\alpha$ is in fact a bijection between $P^O_n$ and $P^E_n$. This finishes our proof.
\end{proof}

\subsection{Connection to large Schr\"oder numbers}\label{sec12b3}

Here, we prove another identity for $s_{d}(n)$, and briefly discuss a version of weighted large Schr\"oder numbers that is analogous to $s_d(n)$. We first have the following:

\begin{proposition}\label{SdIdentity2}
For every real number $d \neq -1$ and every integer $n \geq 2$,
\[
s_d(n) = \frac{(-1)^{n-1} d}{d+1} s_{-d-1}(n).
\]
\end{proposition}

\begin{proof}
Recall the generating function $y = \sum_{n \geq 1} s_d(n)x^n$. We can rewrite~\eqref{s_dGeneratingFunction}  as the functional equation $y = x \left( \frac{1- y}{1-(d+1)y} \right)$, apply Lagrange inversion, and obtain that
\begin{align*}
s_d(n) =[x^n]y &= \frac{1}{n} [y^{n-1}] \left( \frac{1- y}{1-(d+1)y} \right)^{n} \\
&= \frac{1}{n} [y^{n-1}] \left( \sum_{i \geq 0} \binom{n}{i} (-1)^i \right) \left( \sum_{j \geq 0} \binom{n+j-1}{n-1} (d+1)^j \right)\\
&= \frac{1}{n}  \left( \sum_{i \geq 0} (-1)^i \binom{n}{i}   \binom{2n-i-2}{n-1} (d+1)^{n-i-1} \right).
\end{align*}
Since $\binom{2n-i-2}{n-1} = 0$ when $i > n-1$, the sum could simply run from $i=0$ to $i=n-1$. Moreover, by re-indexing the sum using $j=n-i-1$, we have
\begin{align*}
s_d(n) &=  \frac{1}{n}  \left( \sum_{j = 0}^{n-1} (-1)^{n-j-1} \binom{n}{n-j-1}   \binom{2n-(n-j-1)-2}{n-1} (d+1)^{j} \right)\\
&= (-1)^{n-1} \left( \sum_{j = 0}^{n-1} \frac{1}{n} \binom{n}{j+1}   \binom{n+j-1}{n-1} (-d-1)^{j} \right)
\end{align*}
Now notice that
\begin{align*}
\frac{1}{n}  \binom{n}{j+1}   \binom{n+j-1}{n-1} 
&= \frac{1}{n-1} \left( \binom{n-1}{j} \binom{n+j-1}{n} + \binom{n-1}{j+1} \binom{n+j}{n} \right) \\
&= s(n,j) + s(n,j+1).
\end{align*}
Hence, we have
\begin{align*}
s_d(n) &=  (-1)^{n-1}  \sum_{j = 0}^{n-1} (s(n,j) + s(n,j+1)) (-d-1)^{j}\\
&=  (-1)^{n-1}  \left( \sum_{j = 0}^{n-1} s(n,j) (-d-1)^{j} - \frac{1}{d+1} \sum_{j=0}^{n-1}  s(n,j+1) (-d-1)^{j+1} \right)\\
&= (-1)^{n-1} \frac{d}{d+1} s_{-d-1}(n),
\end{align*}
finishing the proof of our claim. Note that in the second equality above we used the fact that $s(n,0) = s(n,n) = 0$ for all $n \geq 2$, as well as the assumption that $d \neq -1$.
\end{proof}

\ignore{
Proposition~\ref{SdIdentity2} readily implies the following identity for $s(n,k)$:

\begin{corollary}\label{SdIdentity2Cor}
For every real number $d \neq 0, -1$ and integer $n \geq 2$,
\[
\sum_{k=0}^{n-1} \left( d^{k-1} +(-1)^{n-k} (d+1)^{k-1} \right) s(n,k) = 0.
\]
\end{corollary}

\begin{proof}
From Proposition~\ref{SdIdentity2}, we have
\begin{align*}
0 &= s_d(n) - (-1)^{n-1} \frac{d}{d+1} s_{-d-1}(n)\\
&= \sum_{k=0}^{n-1} \left( d^k - (-1)^{n-1} \frac{d}{d+1} (-d-1)^k \right) s(n,k) \\
&= \sum_{k=0}^{n-1} \left( d^k + (-1)^{n-k} d(d+1)^{k-1} \right) s(n,k) \\
&= d \sum_{k=0}^{n-1} \left( d^{k-1} + (-1)^{n-k} (d+1)^{k-1} \right) s(n,k),
\end{align*}
which proves our claim whenever $d \neq 0,-1$.
\end{proof}
}
Notice that Proposition~\ref{SdIdentity2} rearranges to $(-1)^{n-1}s_{-d-1}(n) = \frac{d+1}{d} s_d(n)$ (for $d \neq 0,-1$), and that we encountered the expression $s(n,j) + s(n,j+1)$ in its proof. We note that these quantities have a natural connection with the large Schr\"oder numbers. For every integer $n \geq 1$, we define the set of \emph{large Schr\"oder paths} $\overline{\P}_n$ to be the set of lattice paths from $(0,0)$ to $(2n-2,0)$ that satisfy properties (P1) and (P2) (but not necessarily (P3)) in the definition of small Schr\"oder paths. Thus, $\P_n$ is a subset of $\overline{\P}_n$ for every $n$. Figure~\ref{fig2b} illustrates the large Schr\"oder paths in $\overline{\P}_4$ that do not belong to $\P_4$.

\begin{center}
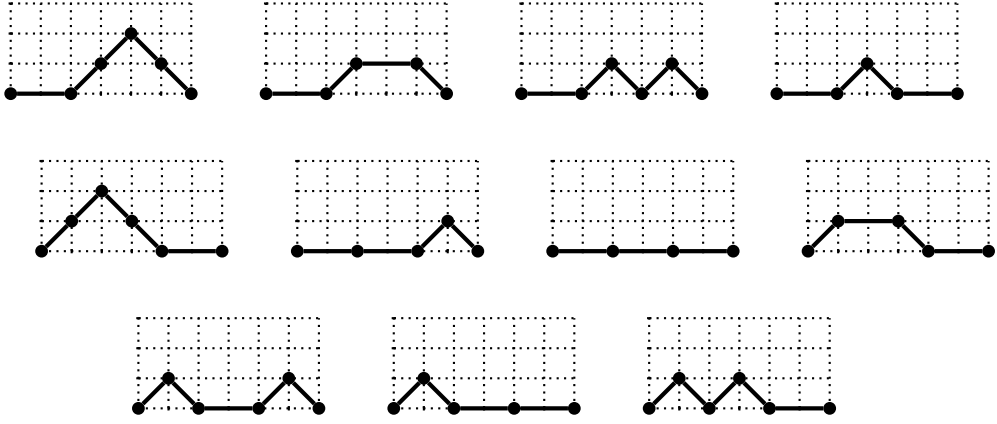
\begin{figure}[h]
\def\fig2scale{0.4}
\[
\begin{array}{c}
\begin{tikzpicture}[scale = \fig2scale, xscale=1,yscale=1, font=\scriptsize\sffamily, thick,main node/.style={circle,inner sep=0.5mm,draw, fill}]
\def\xlb{0}; \def\xub{6}; \def\ylb{0}; \def\yub{3}; \def\buf{0};
\foreach \x in {\xlb ,...,\xub}
{    \ifthenelse{\NOT 0 = \x}{\draw[thick](\x ,-2pt) -- (\x ,2pt);}{}
\draw[dotted, thick](\x,\ylb- \buf) -- (\x,\yub + \buf);}
\foreach \y in {\ylb ,...,\yub}
{    \ifthenelse{\NOT 0 = \y}{\draw[thick](-2pt, \y) -- (2pt, \y);}{}
\draw[dotted, thick](\xlb- \buf, \y) -- (\xub + \buf, \y);}

\node[main node] at (0,0) (0) {};
\node[main node] at (2,0) (1) {};
\node[main node] at (3,1) (2) {};
\node[main node] at (4,2) (3) {};
\node[main node] at (5,1) (4) {};
\node[main node] at (6,0) (5) {};
\draw[ultra thick] (0) -- (1) -- (2) -- (3)-- (4)-- (5);

\end{tikzpicture}
\qquad
\begin{tikzpicture}[scale = \fig2scale, xscale=1,yscale=1, font=\scriptsize\sffamily, thick,main node/.style={circle,inner sep=0.5mm,draw, fill}]
\def\xlb{0}; \def\xub{6}; \def\ylb{0}; \def\yub{3}; \def\buf{0};
\foreach \x in {\xlb ,...,\xub}
{    \ifthenelse{\NOT 0 = \x}{\draw[thick](\x ,-2pt) -- (\x ,2pt);}{}
\draw[dotted, thick](\x,\ylb- \buf) -- (\x,\yub + \buf);}
\foreach \y in {\ylb ,...,\yub}
{    \ifthenelse{\NOT 0 = \y}{\draw[thick](-2pt, \y) -- (2pt, \y);}{}
\draw[dotted, thick](\xlb- \buf, \y) -- (\xub + \buf, \y);}
\node[main node] at (0,0) (0) {};
\node[main node] at (2,0) (1) {};
\node[main node] at (3,1) (2) {};
\node[main node] at (5,1) (3) {};
\node[main node] at (6,0) (4) {};
\draw[ultra thick] (0) -- (1) -- (2) -- (3)-- (4);

\end{tikzpicture}
\qquad
\begin{tikzpicture}[scale = \fig2scale, xscale=1,yscale=1, font=\scriptsize\sffamily, thick,main node/.style={circle,inner sep=0.5mm,draw, fill}]
\def\xlb{0}; \def\xub{6}; \def\ylb{0}; \def\yub{3}; \def\buf{0};
\foreach \x in {\xlb ,...,\xub}
{    \ifthenelse{\NOT 0 = \x}{\draw[thick](\x ,-2pt) -- (\x ,2pt);}{}
\draw[dotted, thick](\x,\ylb- \buf) -- (\x,\yub + \buf);}
\foreach \y in {\ylb ,...,\yub}
{    \ifthenelse{\NOT 0 = \y}{\draw[thick](-2pt, \y) -- (2pt, \y);}{}
\draw[dotted, thick](\xlb- \buf, \y) -- (\xub + \buf, \y);}
\node[main node] at (0,0) (0) {};
\node[main node] at (2,0) (1) {};
\node[main node] at (3,1) (2) {};
\node[main node] at (4,0) (3) {};
\node[main node] at (5,1) (4) {};
\node[main node] at (6,0) (5) {};
\draw[ultra thick] (0) -- (1) -- (2) -- (3)-- (4)-- (5);

\end{tikzpicture}
\qquad
\begin{tikzpicture}[scale = \fig2scale, xscale=1,yscale=1, font=\scriptsize\sffamily, thick,main node/.style={circle,inner sep=0.5mm,draw, fill}]
\def\xlb{0}; \def\xub{6}; \def\ylb{0}; \def\yub{3}; \def\buf{0};
\foreach \x in {\xlb ,...,\xub}
{    \ifthenelse{\NOT 0 = \x}{\draw[thick](\x ,-2pt) -- (\x ,2pt);}{}
\draw[dotted, thick](\x,\ylb- \buf) -- (\x,\yub + \buf);}
\foreach \y in {\ylb ,...,\yub}
{    \ifthenelse{\NOT 0 = \y}{\draw[thick](-2pt, \y) -- (2pt, \y);}{}
\draw[dotted, thick](\xlb- \buf, \y) -- (\xub + \buf, \y);}
\node[main node] at (0,0) (0) {};
\node[main node] at (2,0) (1) {};
\node[main node] at (3,1) (2) {};
\node[main node] at (4,0) (3) {};
\node[main node] at (6,0) (4) {};
\draw[ultra thick] (0) -- (1) -- (2) -- (3)-- (4);

\end{tikzpicture}
\qquad

\\
\\
\begin{tikzpicture}[scale = \fig2scale, xscale=1,yscale=1, font=\scriptsize\sffamily, thick,main node/.style={circle,inner sep=0.5mm,draw, fill}]
\def\xlb{0}; \def\xub{6}; \def\ylb{0}; \def\yub{3}; \def\buf{0};
\foreach \x in {\xlb ,...,\xub}
{    \ifthenelse{\NOT 0 = \x}{\draw[thick](\x ,-2pt) -- (\x ,2pt);}{}
\draw[dotted, thick](\x,\ylb- \buf) -- (\x,\yub + \buf);}
\foreach \y in {\ylb ,...,\yub}
{    \ifthenelse{\NOT 0 = \y}{\draw[thick](-2pt, \y) -- (2pt, \y);}{}
\draw[dotted, thick](\xlb- \buf, \y) -- (\xub + \buf, \y);}
\node[main node] at (0,0) (0) {};
\node[main node] at (1,1) (1) {};
\node[main node] at (2,2) (2) {};
\node[main node] at (3,1) (3) {};
\node[main node] at (4,0) (4) {};
\node[main node] at (6,0) (5) {};
\draw[ultra thick] (0) -- (1) -- (2) -- (3)-- (4)-- (5);

\end{tikzpicture}
\qquad
\begin{tikzpicture}[scale = \fig2scale, xscale=1,yscale=1, font=\scriptsize\sffamily, thick,main node/.style={circle,inner sep=0.5mm,draw, fill}]
\def\xlb{0}; \def\xub{6}; \def\ylb{0}; \def\yub{3}; \def\buf{0};
\foreach \x in {\xlb ,...,\xub}
{    \ifthenelse{\NOT 0 = \x}{\draw[thick](\x ,-2pt) -- (\x ,2pt);}{}
\draw[dotted, thick](\x,\ylb- \buf) -- (\x,\yub + \buf);}
\foreach \y in {\ylb ,...,\yub}
{    \ifthenelse{\NOT 0 = \y}{\draw[thick](-2pt, \y) -- (2pt, \y);}{}
\draw[dotted, thick](\xlb- \buf, \y) -- (\xub + \buf, \y);}
\node[main node] at (0,0) (0) {};
\node[main node] at (2,0) (1) {};
\node[main node] at (4,0) (2) {};
\node[main node] at (5,1) (3) {};
\node[main node] at (6,0) (4) {};
\draw[ultra thick] (0) -- (1) -- (2) -- (3)-- (4);

\end{tikzpicture}
\qquad
\begin{tikzpicture}[scale = \fig2scale, xscale=1,yscale=1, font=\scriptsize\sffamily, thick,main node/.style={circle,inner sep=0.5mm,draw, fill}]
\def\xlb{0}; \def\xub{6}; \def\ylb{0}; \def\yub{3}; \def\buf{0};
\foreach \x in {\xlb ,...,\xub}
{    \ifthenelse{\NOT 0 = \x}{\draw[thick](\x ,-2pt) -- (\x ,2pt);}{}
\draw[dotted, thick](\x,\ylb- \buf) -- (\x,\yub + \buf);}
\foreach \y in {\ylb ,...,\yub}
{    \ifthenelse{\NOT 0 = \y}{\draw[thick](-2pt, \y) -- (2pt, \y);}{}
\draw[dotted, thick](\xlb- \buf, \y) -- (\xub + \buf, \y);}
\node[main node] at (0,0) (0) {};
\node[main node] at (2,0) (1) {};
\node[main node] at (4,0) (2) {};
\node[main node] at (6,0) (3) {};
\draw[ultra thick] (0) -- (1) -- (2) -- (3);

\end{tikzpicture}
\qquad
\begin{tikzpicture}[scale = \fig2scale, xscale=1,yscale=1, font=\scriptsize\sffamily, thick,main node/.style={circle,inner sep=0.5mm,draw, fill}]
\def\xlb{0}; \def\xub{6}; \def\ylb{0}; \def\yub{3}; \def\buf{0};
\foreach \x in {\xlb ,...,\xub}
{    \ifthenelse{\NOT 0 = \x}{\draw[thick](\x ,-2pt) -- (\x ,2pt);}{}
\draw[dotted, thick](\x,\ylb- \buf) -- (\x,\yub + \buf);}
\foreach \y in {\ylb ,...,\yub}
{    \ifthenelse{\NOT 0 = \y}{\draw[thick](-2pt, \y) -- (2pt, \y);}{}
\draw[dotted, thick](\xlb- \buf, \y) -- (\xub + \buf, \y);}
\node[main node] at (0,0) (0) {};
\node[main node] at (1,1) (1) {};
\node[main node] at (3,1) (2) {};
\node[main node] at (4,0) (3) {};
\node[main node] at (6,0) (4) {};
\draw[ultra thick] (0) -- (1) -- (2) -- (3)-- (4);

\end{tikzpicture}
\\
\\
\begin{tikzpicture}[scale = \fig2scale, xscale=1,yscale=1, font=\scriptsize\sffamily, thick,main node/.style={circle,inner sep=0.5mm,draw, fill}]
\def\xlb{0}; \def\xub{6}; \def\ylb{0}; \def\yub{3}; \def\buf{0};
\foreach \x in {\xlb ,...,\xub}
{    \ifthenelse{\NOT 0 = \x}{\draw[thick](\x ,-2pt) -- (\x ,2pt);}{}
\draw[dotted, thick](\x,\ylb- \buf) -- (\x,\yub + \buf);}
\foreach \y in {\ylb ,...,\yub}
{    \ifthenelse{\NOT 0 = \y}{\draw[thick](-2pt, \y) -- (2pt, \y);}{}
\draw[dotted, thick](\xlb- \buf, \y) -- (\xub + \buf, \y);}
\node[main node] at (0,0) (0) {};
\node[main node] at (1,1) (1) {};
\node[main node] at (2,0) (2) {};
\node[main node] at (4,0) (3) {};
\node[main node] at (5,1) (4) {};
\node[main node] at (6,0) (5) {};
\draw[ultra thick] (0) -- (1) -- (2) -- (3)-- (4)-- (5);

\end{tikzpicture}
\qquad
\begin{tikzpicture}[scale = \fig2scale, xscale=1,yscale=1, font=\scriptsize\sffamily, thick,main node/.style={circle,inner sep=0.5mm,draw, fill}]
\def\xlb{0}; \def\xub{6}; \def\ylb{0}; \def\yub{3}; \def\buf{0};
\foreach \x in {\xlb ,...,\xub}
{    \ifthenelse{\NOT 0 = \x}{\draw[thick](\x ,-2pt) -- (\x ,2pt);}{}
\draw[dotted, thick](\x,\ylb- \buf) -- (\x,\yub + \buf);}
\foreach \y in {\ylb ,...,\yub}
{    \ifthenelse{\NOT 0 = \y}{\draw[thick](-2pt, \y) -- (2pt, \y);}{}
\draw[dotted, thick](\xlb- \buf, \y) -- (\xub + \buf, \y);}
\node[main node] at (0,0) (0) {};
\node[main node] at (1,1) (1) {};
\node[main node] at (2,0) (2) {};
\node[main node] at (4,0) (3) {};
\node[main node] at (6,0) (4) {};
\draw[ultra thick] (0) -- (1) -- (2) -- (3)-- (4);

\end{tikzpicture}
\qquad
\begin{tikzpicture}[scale = \fig2scale, xscale=1,yscale=1, font=\scriptsize\sffamily, thick,main node/.style={circle,inner sep=0.5mm,draw, fill}]
\def\xlb{0}; \def\xub{6}; \def\ylb{0}; \def\yub{3}; \def\buf{0};
\foreach \x in {\xlb ,...,\xub}
{    \ifthenelse{\NOT 0 = \x}{\draw[thick](\x ,-2pt) -- (\x ,2pt);}{}
\draw[dotted, thick](\x,\ylb- \buf) -- (\x,\yub + \buf);}
\foreach \y in {\ylb ,...,\yub}
{    \ifthenelse{\NOT 0 = \y}{\draw[thick](-2pt, \y) -- (2pt, \y);}{}
\draw[dotted, thick](\xlb- \buf, \y) -- (\xub + \buf, \y);}
\node[main node] at (0,0) (0) {};
\node[main node] at (1,1) (1) {};
\node[main node] at (2,0) (2) {};
\node[main node] at (3,1) (3) {};
\node[main node] at (4,0) (4) {};
\node[main node] at (6,0) (5) {};
\draw[ultra thick] (0) -- (1) -- (2) -- (3)-- (4)-- (5);

\end{tikzpicture}
\qquad

\end{array}
\]

\caption{The $11$ large Schr\"oder paths in $\overline{\P}_4 \setminus \P_4$}\label{fig2b}
\end{figure}
\end{center}

Let $\overline{s}(n) = |\overline{\P}_n|$. Then $\overline{s}(n)$ gives the large Schr\"oder numbers (\seqnum{A006318} in the OEIS):

\begin{center}
\begin{tabular}{l|rrrrrrrrr}
$n$ & 1 & 2 & 3 & 4 & 5 & 6 & 7 & 8 & $\cdots$ \\
\hline
$\overline{s}(n)$ & 1 & 2 & 6 & 22 & 90 & 394 & 1806 & 8558 & $\cdots$ \\
\end{tabular}
\end{center}

It is well known that $\overline{s}(n) = 2s(n)$ for every $n \geq 2$. To see this, consider a path $P \in \overline{\P_n} \setminus \P_n$, and write $P = P_1FP_2$ where $P_1$ does not contain a flat step on the $x$-axis. Then it is easy to show that the mapping $\beta : \overline{\P}_n \setminus \P_n  \to \P_n$ where $\beta(P) = P_1UP_2D$ is a bijection, thus showing that $|\overline{\P}_n| = 2 |\P_n|$. Moreover, notice that $\beta(P)$ has exactly one more up step than $P$. Thus, if we let $\overline{\P}_{n,k}$ denote the set of paths in $\overline{\P}_n$ with exactly $k$ up steps, then we have
\[
\overline{\P}_{n,k} = \P_{n,k} \cup \set{ P : \beta(P) \in \P_{n,k+1}}.
\]
The sets $\P_{n,k}$ and $\set{ P : \beta(P) \in \P_{n,k+1}}$ are obviously disjoint. Thus, if we let $\overline{s}(n,k) = | \overline{\P}_{n,k}|$, we have
\[
\overline{s}(n,k) = s(n,k) + s(n,k+1)
\]
for every $n \geq 2$ and every $k \in \set{0, \ldots, n-1}$. Furthermore, given a real number $d$, if we define the weighted large Schr\"oder number $\overline{s}_d(n) = \sum_{k=0}^{n-1} \overline{s}(n,k) d^k$, then we see that
\[
\overline{s}_d(n) = \sum_{k=0}^{n-1} \overline{s}(n,k) d^k = \sum_{k=0}^{n-1} (s(n,k) + s(n,k+1)) d^k = \frac{d+1}{d} s_d(n).
\]
Thus, Proposition~\ref{SdIdentity2} can be written as simply
\[
\overline{s}_d(n) = (-1)^{n-1} s_{-d-1}(n)
\]
for all $d \neq 0,-1$ and $n \geq 2$. We will revisit Proposition~\ref{SdIdentity2} from the perspective of Dyck paths in the next section.

\section{$s_d(n)$ in terms of Dyck paths}\label{sec13}

We next discuss several families of Dyck paths that are counted by $s_d(n)$. Recall that a Dyck path is a small Schr\"oder path with only up and down steps (i.e., no flat steps). Let  $\D_n$ be the set of Dyck paths that starts at $(0,0)$ and ends at $(2n-2,0)$. (Equivalently, $\D_n = \P_{n,n-1}$.) Also, it will be convenient to have the following notation for our subsequent discussion: Given a path $P \in \P_n$, let
\begin{itemize}
\item
$U(P)$, $F(P)$, and $D(P)$ respectively denote the number of up, flat, and down steps in $P$;
\item
$V(P)$ denote the number valleys (i.e., occurrences of $DU$) in $P$;
\item
$K(P)$ denote the number peaks (i.e., occurrences of $UD$) in $P$;
\item
$U_V(P)$ (resp., $D_V(P)$) denote the number of up (resp., down) steps in $P$ that is not contained in a valley;
\item
$U_K(P)$ (resp., $D_K(P)$) denote the number of up (resp., down) steps in $P$ that is not contained in a peak.
\end{itemize}

Recently, Chen and Pan~\cite{ChenP17} studied the following notion of weighted Catalan numbers: Given real numbers $a$ and $b$, define
\begin{equation}\label{D_nab}
c^V_{a,b}(n) = \sum_{P \in \D_n} a^{U_V(P)} b^{V(P)}.
\end{equation}
Then $c^V_{1,1}(n)$ gives the ordinary Catalan numbers $c(n)$. Interestingly, $c^V_{a,b}(n)$ also coincides with $s_d(n)$ for certain choices of $a$ and $b$.

\begin{proposition}\label{s_dValley}
For every real number $d$ and integer $n \geq 1$,
\[
s_d(n) = c^V_{d, d+1}(n).
\]
\end{proposition}

\begin{proof}
Given a path $P \in \D_n$, let $S \subseteq \set{1, \ldots, V(P)}$ be an arbitrary subset of the valleys of $P$. Define the function $f_S : \D_n \to \P_n$ such that $f_S(P)$ is the path obtained from replacing every valley not in $S$ by a flat step $F$. Observe that by this construction, the resulting path $f_S(P)$ is a small Schr\"oder path with $U_V(P) + |S|$ up steps.

Now observe that, for every path $P \in \D_n$,
\begin{equation}\label{s_dValley1}
d^{U_V(P)} (d+1)^{V(P)} = \sum_{S \subseteq \set{1, \ldots, V(P)}} d^{U_V(P) + |S|} =  \sum_{S \subseteq \set{1, \ldots, V(P)}} d^{U(f_S(P))}
\end{equation}
Also, every path  $Q \in \P_n$ can be written as $f_S(P)$ for a unique choice of $P \in \D_n$ and subset $S$ (namely, let $P$ be the path with all flat steps in $Q$ replaced by valleys, and let $S$ be the set of valleys in $P$ that are valleys in $Q$). Thus, if we sum over all paths in $\D_n$ on both sides of~\eqref{s_dValley1}, we obtain
\[
c^V_{d,d+1}(n) = \sum_{P \in \D_n} d^{U_V(P)} (d+1)^{V(P)} = \sum_{Q \in \P_n} d^{U(Q)} = s_d(n).
\]
\end{proof}

We remark that Proposition~\ref{s_dValley} also follows from~\cite[eq.\ (1.15)]{ChenP17} and our discussion in Section~\ref{sec12b3} showing that $\overline{s}_d(n) = \frac{d+1}{d} s_d(n)$. Also, notice that when $d$ is a positive integer, we obtain that $s_d(n)$ counts the number ways to construct a Dyck path from $(0,0)$ to $(2n-2,0)$ where each valley can be painted one of $d+1$ colors, and each up step that does not belong to a valley can be painted one of $d$ colors. Moreover, it is obvious that 
\[
c^V_{-1,0}(n) = \sum_{P \in \D_n} (-1)^{U_V(P)} (0)^{V(P)} = (-1)^{n-1}
\]
 for ever $n \geq 1$, since the only Dyck path in $\D_n$ with no valleys is the path $U^{n-1} D^{n-1}$. Thus, in the case of $d=-1$, Proposition~\ref{s_dValley} implies Proposition~\ref{s_MinusOne}.
 
Next, we show that $c^V_{a,b}(n)$ can be related to $s_d(n)$ even when $b \neq a+1$.

\begin{corollary}\label{c^Vab}
For every integer $n \geq 1$ and  every real number $a,b$ where $a \neq b$,
\[
c^V_{a,b}(n) = (b-a)^{n-1} s_{a/(b-a)}(n).
\]
\end{corollary}

\begin{proof}
\begin{align*}
c^V_{a,b}(n) &= \sum_{P \in \D_n} a^{U_V(P)} b^{V(P)} \\
&= (b-a)^{n-1} \sum_{P \in \D_n} \left( \frac{a}{b-a} \right)^{U_V(P)} \left(\frac{a}{b-a} +1\right)^{V(P)} \\
&= (b-a)^{n-1} c^V_{a/(b-a). a/(b-a) +1}(n)\\
&= (b-a)^{n-1} s_{a/(b-a)}(n).
\end{align*}
Note that the second equality follows from the fact that $U_V(P) + V(P) = U(P) = n-1$ for all $P \in \D_n$.
\end{proof}

Corollary~\ref{c^Vab} readily implies that certain families of small Schr\"oder paths and Dyck paths have the same size. For instance:

\begin{corollary}\label{s_dValleyCor}
For all integers $m,n\geq 1$, the following quantities are both equal to $m^{n-1} s_{1/m}(n)$:
\begin{itemize}
\item[(i)]
the number of small Schr\"oder paths in $\P_n$ where each flat step can be painted one of $m$ colors;
\item[(ii)]
the number of Dyck paths in $\D_n$ where each valley can be painted one of $m+1$ colors.
\end{itemize}
\end{corollary}

\begin{proof}
From Corollary~\ref{c^Vab}, we obtain that 
\[
m^{n-1} s_{1/m}(n) = c^V_{1,m+1}(n),
\]
showing (ii). For (i),
\[
m^{n-1} s_{1/m}(n) = m^{n-1} \sum_{P \in \P_n} (1/m)^{U(P)} 1^{F(P)} = \sum_{P \in \P_n} m^{F(P)},
\]
since $U(P) + F(P) = n-1$ for every $P \in \P_n$. This finishes the proof.
\end{proof}

We next study another variant of weighted Catalan numbers that, as we shall see, is equal to $c^V_{a,b}(n)$. Given real numbers $a$ and $b$, define
\[
c^K_{a,b}(n) = \sum_{P \in \D_n} a^{K(P)} b^{U_K(P)}.
\]
Then we have the following:

\begin{proposition}\label{s_dPeak}
For all real numbers $a,b$ and integer $n \geq 1$
\[
c^V_{a,b}(n) = c^K_{a,b}(n).
\]
\end{proposition}

To prove Proposition~\ref{s_dPeak}, we need another tree-to-path procedure. Recall the mapping $\Psi$ described in Definition~\ref{defnTreeToPath}, where we turn a tree into a path by walking counterclockwise around the tree starting at the root. The following mapping $\Psi'$ is very similar to $\Psi$, except we walk clockwise around the tree this time.

\begin{definition}\label{defnTreeToPath2}
Given a Schr\"oder tree $T \in \T_n$, construct the path $\Psi'(T) \in \P_n$ as follows:
\begin{enumerate}
\item
Suppose $T$ has $q$ nodes. We perform an postorder traversal of $T$ in reverse (i.e., we recursively traverse the tree in the order of root, then right, then left), and label the nodes $a_1, a_2, \ldots, a_q$ in that order.
\item
Notice that $a_1$ must be the root of $T$. For every $i \geq 2$, define the function
\[
\psi'(a_i) = 
\begin{cases}
U & \tn{if $a_i$ is the rightmost child of its parent;}\\
D & \tn{if $a_i$ is the leftmost child of its parent;}\\
F & \tn{otherwise.}
\end{cases}
\]
\item
Define
\[
\Psi'(T) =\psi'(a_2) \psi'(a_3)  \cdots \psi'(a_q).
\]
\end{enumerate}

\end{definition}

Figure~\ref{fig3b} illustrates the mapping $\Psi'$ applied to the same tree shown in Figure~\ref{fig3}.

\begin{center}
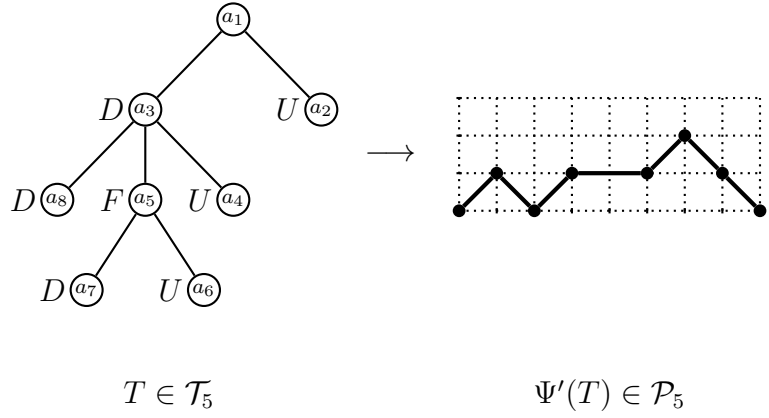
\begin{figure}[h]
\[
\begin{array}{ccc}

\begin{tikzpicture}
[scale=0.6,xscale=1.3, thick,main node/.style={circle,inner sep=0.3mm,draw,font=\scriptsize\sffamily}]
  \node[main node] at (3.5,6) (0) {$a_1$};
    \node[main node, label={[label distance=-0.1cm]180:$U$}] at (5,4) (r) {$a_2$};
    \node[main node, label={[label distance=-0.1cm]180:$D$}] at (2,4) (l) {$a_3$};
    \node[main node, label={[label distance=-0.1cm]180:$U$}] at (3.5,2) (l3){$a_4$};
    \node[main node, label={[label distance=-0.1cm]180:$F$}] at (2,2) (l2){$a_5$};
    \node[main node, label={[label distance=-0.1cm]180:$U$}] at (3,0) (l22){$a_6$};
    \node[main node, label={[label distance=-0.1cm]180:$D$}] at (1,0) (l21){$a_7$};
    \node[main node, label={[label distance=-0.1cm]180:$D$}] at (0.5,2) (l1){$a_8$};

  \path[every node/.style={font=\sffamily}]
    (0) edge (l)
    (0) edge (r)
                (l) edge (l1)
                 (l) edge (l2)
            (l) edge (l3)
(l2) edge (l21)
(l2) edge (l22);
\end{tikzpicture}

&
\raisebox{2cm}{$\longrightarrow$}
&
\raisebox{1.25cm}{
\begin{tikzpicture}[scale = 0.5, xscale=1,yscale=1, font=\scriptsize\sffamily, thick,main node/.style={circle,inner sep=0.5mm,draw, fill}]
\def\xlb{0}; \def\xub{8}; \def\ylb{0}; \def\yub{3}; \def\buf{0};
\foreach \x in {\xlb ,...,\xub}
{    \ifthenelse{\NOT 0 = \x}{\draw[thick](\x ,-2pt) -- (\x ,2pt);}{}
\draw[dotted, thick](\x,\ylb- \buf) -- (\x,\yub + \buf);}
\foreach \y in {\ylb ,...,\yub}
{    \ifthenelse{\NOT 0 = \y}{\draw[thick](-2pt, \y) -- (2pt, \y);}{}
\draw[dotted, thick](\xlb- \buf, \y) -- (\xub + \buf, \y);}

\node[main node] at (0,0) (0) {};
\node[main node] at (1,1) (1) {};
\node[main node] at (2,0) (2) {};
\node[main node] at (3,1) (3) {};
\node[main node] at (5,1) (4) {};
\node[main node] at (6,2) (5) {};
\node[main node] at (7,1) (6) {};
\node[main node] at (8,0) (7) {};
\draw[ultra thick] (0) -- (1) -- (2) -- (3)-- (4)-- (5)-- (6)-- (7);

\end{tikzpicture}
}
\\
\\
T \in \T_5 & & \Psi'(T) \in \P_5\\
\end{array}
\]
\caption{Illustrating the mapping $\Psi'$ (Definition~\ref{defnTreeToPath2})}\label{fig3b}
\end{figure}
\end{center}

For essentially the same reasons why $\Psi : \T_n \to \P_n$ is a bijection, $\Psi'$ is a bijection between these two sets as well. Moreover, both $\Psi, \Psi'$ are in fact bijections between $\T_{n,n-1}$ (which are necessarily full binary trees) and $\P_{n,n-1} = \D_n$. The follow result relates various statistics of the Dyck paths obtained from applying $\Psi$ and $\Psi'$ to the same binary tree.

\begin{lemma}\label{psiPsiPrime}
For every tree $T \in \T_{n,n-1}$ where $n \geq 2$, we have
\begin{align}
\label{psiPsiPrime1}& K( \Psi(T)) + K( \Psi'(T)) = n,\\
\label{psiPsiPrime2}& K( \Psi(T))   = U_V( \Psi'(T)),\\
\label{psiPsiPrime3}& U_K( \Psi(T))   = V( \Psi'(T)).
\end{align}
\end{lemma}

\begin{proof}
We first prove~\eqref{psiPsiPrime1} by induction on $n$. When $n=2$, there is a unique binary tree $T$ in $\T_{2,1}$ and it is easy to check that $K(\Psi(T)) + K(\Psi'(T)) = 1+1=2$, so the base case holds. For the inductive step, suppose the root of $T$ has left subtree $T_1$ and right subtree $T_2$, where $T_1, T_2$ have respectively $n_1$ and $n_2$ leaves.

Now observe that $\Psi(T) = U \Psi(T_1) D \Psi(T_2)$. Since both $\Psi(T_1)$ and $\Psi(T_2)$ are Dyck paths in their own rights, each must either start with a $U$ and end with a $D$, or is empty. Hence, we obtain that $K( \Psi(T)) = K(\Psi(T_1)) + K(\Psi(T_2))$.

Similarly, we see that $\Psi'(T) = U \Psi'(T_2) D \Psi'(T_1)$, and  $K(\Psi'(T)) =  K(\Psi'(T_1)) + K(\Psi'(T_2))$. Thus, using the inductive hypothesis, we obtain
\[
K(\Psi(T)) + K(\Psi'(T)) = K(\Psi(T_1)) + K(\Psi(T_2)) +  K(\Psi'(T_1)) + K(\Psi'(T_2)) = n_1 + n_2 = n.
\]
This proves~\eqref{psiPsiPrime1}.

We next prove~\eqref{psiPsiPrime2} and~\eqref{psiPsiPrime3}. Notice that $V(P) = K(P) - 1$ for every $P \in \D_n$. Thus, applying~\eqref{psiPsiPrime1} gives
\[
K(\Psi(T)) + V(\Psi'(T)) = n-1.
\]
Furthermore, since $U(\Psi(T)) = U(\Psi'(T)) = n-1$, we have
\begin{align*}
U_K(\Psi(T)) + K(\Psi(T)) &= n-1,\\
U_V(\Psi'(T)) + V(\Psi'(T)) &= n-1.
\end{align*}
This implies that $K(\Psi(T)) = U_V(\Psi'(T))$ and $U_K(\Psi(T)) = V(\Psi'(T))$, as desired.
\end{proof}

We are now ready to prove Proposition~\ref{s_dPeak}.

\begin{proof}[Proof of Proposition~\ref{s_dPeak}]
Define the mapping $\gamma : \D_n \to \D_n$ where $\gamma(P) = \Psi' ( \Psi^{-1}(P))$. Since $\Psi, \Psi' : \T_{n,n-1} \to \D_n$ are both bijections, $\gamma$ is a bijection between $\D_n$ and itself. Hence,
\[
c^K_{a,b}(n) = \sum_{P \in \D_n} a^{K(P)} b^{U_K(P)} =  \sum_{P \in \D_n} a^{U_V(\gamma(P))} b^{V(\gamma(P))} = c^V_{a,b}(n),
\]
as claimed. Note that the second equality follows from Lemma~\ref{psiPsiPrime}.
\end{proof}

We next point out how Propositions~\ref{s_MinusOneHalf} and~\ref{SdIdentity2} in Section~\ref{sec12b} can be alternatively shown using the connections between $s_d(n), c^V_{a,b}(n)$, and $c^K_{a,b}(n)$ established above. First, we revisit Propsition~\ref{s_MinusOneHalf} and the sequence $s_{-1/2}(n)$. For every integer $n \geq 2$ and $k \in \set{0,\ldots, n-1}$, define
\[
c(n,k) = \frac{1}{n-1} \binom{n-1}{k-1}\binom{n-1}{k},
\]
and let $c(1,0) = 1$. These are known as the Narayana numbers (\seqnum{A090181} in the OEIS), and $c(n,k)$ counts the number of Dyck paths in $\D_n$ with exactly $k$ up steps.

\begin{center}
\begin{tabular}{r|rrrrrrr}
$c(n,k)$ & $k= 0$ & 1 & 2 & 3 & 4 & 5 & $\cdots$ \\
\hline
$n=1$  & 1 & & & & & &\\
2  &0 &1 & & & & &\\
3  &0 & 1& 1& & & &\\
4  &0 &1 & 3& 1& && \\
5  &0 & 1&6 &6 &1 && \\
6  &0  &1 &10 &20 &10 &1 &\\
$\vdots$  & $\vdots$ &$\vdots$ &$\vdots$ &$\vdots$ &$\vdots$ &$\vdots$ &$ \ddots$ \\
\end{tabular}
\end{center}

If we define the generating function
\[
C_d(x) = \sum_{n \geq 1} \sum_{ k \geq 0} c(n,k) d^k x^n,
\]
then it is not hard to check that $C_d(x)$ satisfies the functional equation
\[
C_d(x)^2 + (dx-x-1)C_d(x) + x = 0.
\]
(See, for instance,~\cite[Exercise 6.36b]{Stanley99} for the details of the derivation.) Solving the above gives
\[
C_d(x) = \frac{ 1+x-dx - \sqrt{ (dx-x-1)^2 - 4x}}{2}.
\]
When $d= -1, 1$, this specializes to
\[
C_1(x) = \frac{1-\sqrt{1-4x}}{2}, C_{-1}(x) = \frac{1+2x- \sqrt{1+4x^2}}{2}.
\]
Obviously, $C_1(x)$ is the generating function for the ordinary Catalan numbers, and so $[x^n]C_1(x) = c(n)$ for every $n \geq 1$. Also notice that $C_1(-x^2) = C_{-1}(x) - x$. Putting things together, we see that, for every $n \geq 2$,
\begin{align*}
2^{n-1} s_{-1/2}(n) &= c^K_{-1,1}(n) =\sum_{P \in \D_n} (-1)^{K(P)} = \sum_{P \in \D_n} c(n,k)(-1)^k \\
&= [x^n] C_{-1}(x) = [x^n] C_1(-x^2)  \\
&=
\begin{cases}
0 & \tn{if $n$ is odd;}\\
c(n/2) & \tn{if $n \equiv 0$ (mod $4$);}\\
-c(n/2) & \tn{if $n \equiv 2$ (mod $4$,}
\end{cases}
\end{align*}
which aligns with Proposition~\ref{s_MinusOneHalf}. Moreover, the fact that $2^{n-1}s_{-1/2}(n) = c^K_{-1,1}(n)$ also implies that
\[
2^{n-1}s_{-1/2}(n) = \sum_{k~\tn{even}} c(n,k) + \sum_{k~\tn{odd}} c(n,k)
\]
for all $n \geq 1$. Thus, we see that given a fixed length, the number of Dyck paths with an odd number of peaks and that with an even number of peaks are either identical, or differ by a Catalan number. This produces the following two sequences:

\begin{center}
\begin{tabular}{l|rrrrrrrrrrr}
$n$ & 1 & 2 & 3 & 4 & 5 & 6 & 7 & 8 & 9 & 10 & $\cdots$ \\
\hline
$\left| \set{P \in \D_n : K(P)~\tn{is even}} \right| $ & 0 & 0 & 1 & 3 & 7 & 20 & 66 & 217 & 715 & 2424 & $\cdots$\\
$\left|\set{ P \in \D_n : K(P)~\tn{is odd}} \right| $ & 1 & 1 & 1 & 2 & 7 & 22 & 66 & 212& 715 & 2438& $\cdots$\\
\end{tabular}
\end{center}

These two sequences, with the terms for $n=1$ removed, give~\seqnum{A071688} and~\seqnum{A071684} in the OEIS, respectively.

Next, we re-consider Proposition~\ref{SdIdentity2}. Notice that for every real number $m \neq 1$,
\[
c^K_{m,1}(n) = \sum_{P \in \D_n} m^{K(P)} = m \sum_{P \in \D_n} m^{V(P)} = m c^V_{1,m}(n).
\]
Thus, applying Corollary~\ref{c^Vab} and substituting $d = \frac{1}{m-1}$, we obtain
\[
m c^V_{1,m}(n) = m(m-1)^{n-1} s_{1/(m-1)}(n) = \left( \frac{d+1}{d}\right) d^{-(n-1)} s_{d}(n)
\]
Likewise, we also have
\[
c^K_{m,1}(n) = (1-m)^{n-1} s_{m/(1-m)}(n) = (-1)^{n-1} d^{-(n-1)} s_{-d-1}(n),
\]
which yields a proof of Proposition~\ref{SdIdentity2} for all cases where $d \neq 0$.

Finally, inspired by David Scambler's comments in the OEIS, we show how $s_d(n)$ counts families of Dyck paths with certain forbidden peak types. Scambler claimed (in \seqnum{A107841}, slightly paraphrased here) that $s_2(n)$ counts the number of Dyck paths from $(0,0)$ to $(2n-2,0)$ with 3 types of up steps $U_1, U_2, U_3$, one type of down step $D$, and avoid $U_1D$. Similarly, he commented (\seqnum{A131763}) that $s_3(n)$ gives the number of Dyck paths from $(0,0)$ to $(2n-2,0)$ with two types of up steps $U_1, U_2$, two types of down steps $D_1, D_2$, and avoid $U_1D_1$. Independently, Geffner and Noy~\cite[Theorem 3]{GeffnerN17} showed that $s(n)$ counts the number of Dyck paths in $\D_n$ with one type of up step $U$, two types of down steps $D_1, D_2$, and avoid $UD_1$. We show that a generalization of these statements follow readily from our findings above.

\begin{proposition}\label{s_dScambler}
Given integers $k,\ell \geq 1$, let $D_{k,\ell}(n)$ denote the number of ways to construct Dyck paths from $(0,0)$ to $(2n-2,0)$ with $k$ types of up steps $U_1, \ldots, U_k$, $\ell$ types of down steps $D_1, \ldots, D_{\ell}$, and avoid peaks of the type $U_1D_1$. Then 
\[
D_{k,\ell}(n) = s_{k\ell -1}(n).
\]
for every $n \geq 1$.
\end{proposition}

\begin{proof}
Notice that $D_{k, \ell}(n)$ can be reinterpreted as the number of Dyck paths where
\begin{itemize}
\item
each of the $K(P)$ peaks of $P$ is assigned $k \ell - 1$ colors (since there are a total of $k \ell$ types of possible peaks, $1$ of which is forbidden);
\item
each of the $U_K(P)$ up steps not contained in a peak is assigned one of  $k$ colors;
\item
each of the $D_K(P)$ down steps not contained in a peak is assigned one of $\ell$ colors.
\end{itemize}
Hence, we obtain that
\begin{align*}
D_{k,\ell}(n) 
&=  \sum_{P \in \D_n} k^{U_K(P)} (k\ell-1)^{K(P)} \ell^{D_K(P)}\\
&=  \sum_{P \in \D_n}  (k \ell)^{U_K(P)} (k\ell-1)^{K(P)} \\
&= c^K_{ k\ell -1, k \ell}(n) \\
&= c^V_{ k\ell -1, k \ell}(n) \\
&= s_{k\ell-1}(n).
\end{align*}
Note that the second equality above is due to $U_K(P) = D_K(P)$ for every Dyck path $P$, and the last two equalities are due to Propositions~\ref{s_dValley} and~\ref{s_dPeak}, respectively.
\end{proof}

Note that the above argument can be easily extended to cases where more than $1$ of the $k\l$ possible peak types are forbidden. More precisely, one could show that the number of Dyck paths from $(0,0)$ to $(2n-2,0)$ with $k$ types of up steps, $\l$ types of down steps, with $p$ of the $k\ell$ peak types forbidden, is $c^K_{k\ell -p,  k \ell}(n) = p^{n-1} s_{(k\ell-p)/p}(n)$.

\section{Concluding Remarks}\label{sec14}

In this manuscript, we looked at $s_d(n)$, a natural generalization of the small Schr\"oder numbers that had made cameos in a number of contexts but was never the main focus of study until now. We also saw that the study of $s_d(n)$ led to implications for familiar combinatorial objects such as Schr\"oder paths and Dyck paths.

We remark that there are many natural combinatorial interpretations of $s_d(n)$ that we have yet to mention. For instance, as we have seen with Schr\"oder trees and small Schr\"oder paths, any combinatorial interpretation of $s(n,k)$ readily extends to a corresponding set of objects counted by $s_d(n)$ for positive integers $d$. For another example, it is known~\cite{Stanley96} that $s(n,k)$ counts the number of ways to subdivide the regular $(n+1)$-gon into $k$ regions using non-crossing diagonals. Therefore, $s_d(n)$ gives the number of ways to subdivide the regular $(n+1)$-gon and color each of the regions using one of $d$ colors. 

\begin{center}
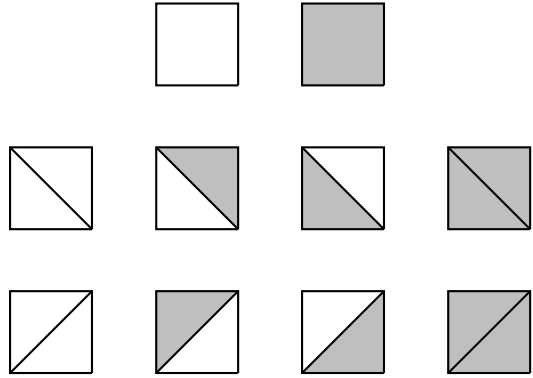
\begin{figure}[h]

\[
\begin{array}{c}
\begin{tikzpicture}[scale = 0.77]
\draw[thick]
({cos(315)}, {sin(315)}) --({cos(45)}, {sin(45)}) --({cos(135)}, {sin(135)}) --({cos(225)}, {sin(225)}) --({cos(315)}, {sin(315)});
\end{tikzpicture}
\qquad 
\begin{tikzpicture}[scale = 0.77]
\filldraw[lightgray] 
({cos(315)}, {sin(315)}) --({cos(45)}, {sin(45)}) --({cos(135)}, {sin(135)}) --({cos(225)}, {sin(225)}) --({cos(315)}, {sin(315)});
\draw[thick]
({cos(315)}, {sin(315)}) --({cos(45)}, {sin(45)}) --({cos(135)}, {sin(135)}) --({cos(225)}, {sin(225)}) --({cos(315)}, {sin(315)});
\end{tikzpicture}
\\
\\
\begin{tikzpicture}[scale = 0.77]
\draw[thick]
({cos(315)}, {sin(315)}) --({cos(45)}, {sin(45)}) --({cos(135)}, {sin(135)}) --({cos(225)}, {sin(225)}) --({cos(315)}, {sin(315)});
\draw[thick]
({cos(315)}, {sin(315)}) -- ({cos(135)}, {sin(135)});
\end{tikzpicture}
\qquad 
\begin{tikzpicture}[scale = 0.77]
\filldraw[lightgray] 
({cos(315)}, {sin(315)}) -- ({cos(135)}, {sin(135)}) -- ({cos(45)}, {sin(45)}) -- ({cos(315)}, {sin(315)});

\draw[thick]
({cos(315)}, {sin(315)}) --({cos(45)}, {sin(45)}) --({cos(135)}, {sin(135)}) --({cos(225)}, {sin(225)}) --({cos(315)}, {sin(315)});
\draw[thick]
({cos(315)}, {sin(315)}) -- ({cos(135)}, {sin(135)});
\end{tikzpicture}
\qquad 
\begin{tikzpicture}[scale = 0.77]
\filldraw[lightgray] 
({cos(315)}, {sin(315)}) -- ({cos(135)}, {sin(135)}) -- ({cos(225)}, {sin(225)}) -- ({cos(315)}, {sin(315)});
\draw[thick]
({cos(315)}, {sin(315)}) --({cos(45)}, {sin(45)}) --({cos(135)}, {sin(135)}) --({cos(225)}, {sin(225)}) --({cos(315)}, {sin(315)});
\draw[thick]
({cos(315)}, {sin(315)}) -- ({cos(135)}, {sin(135)});
\end{tikzpicture}
\qquad 
\begin{tikzpicture}[scale = 0.77]
\filldraw[lightgray] 
({cos(315)}, {sin(315)}) --({cos(45)}, {sin(45)}) --({cos(135)}, {sin(135)}) --({cos(225)}, {sin(225)}) --({cos(315)}, {sin(315)});

\draw[thick]
({cos(315)}, {sin(315)}) --({cos(45)}, {sin(45)}) --({cos(135)}, {sin(135)}) --({cos(225)}, {sin(225)}) --({cos(315)}, {sin(315)});
\draw[thick]
({cos(315)}, {sin(315)}) -- ({cos(135)}, {sin(135)});
\end{tikzpicture}
\\
\\
\begin{tikzpicture}[scale = 0.77]
\draw[thick]
({cos(315)}, {sin(315)}) --({cos(45)}, {sin(45)}) --({cos(135)}, {sin(135)}) --({cos(225)}, {sin(225)}) --({cos(315)}, {sin(315)});
\draw[thick]
({cos(45)}, {sin(45)}) -- ({cos(225)}, {sin(225)});
\end{tikzpicture}
\qquad
\begin{tikzpicture}[scale = 0.77]
\filldraw[lightgray] 
({cos(45)}, {sin(45)}) --({cos(225)}, {sin(225)}) -- ({cos(135)}, {sin(135)}) -- ({cos(45)}, {sin(45)});
\draw[thick]
({cos(315)}, {sin(315)}) --({cos(45)}, {sin(45)}) --({cos(135)}, {sin(135)}) --({cos(225)}, {sin(225)}) --({cos(315)}, {sin(315)});
\draw[thick]
({cos(45)}, {sin(45)}) -- ({cos(225)}, {sin(225)});
\end{tikzpicture}
\qquad
\begin{tikzpicture}[scale = 0.77]
\filldraw[lightgray] 
({cos(45)}, {sin(45)}) --({cos(225)}, {sin(225)}) -- ({cos(315)}, {sin(315)}) -- ({cos(45)}, {sin(45)});
\draw[thick]
({cos(315)}, {sin(315)}) --({cos(45)}, {sin(45)}) --({cos(135)}, {sin(135)}) --({cos(225)}, {sin(225)}) --({cos(315)}, {sin(315)});
\draw[thick]
({cos(45)}, {sin(45)}) -- ({cos(225)}, {sin(225)});
\end{tikzpicture}
\qquad
\begin{tikzpicture}[scale = 0.77]
\filldraw[lightgray] 
({cos(315)}, {sin(315)}) --({cos(45)}, {sin(45)}) --({cos(135)}, {sin(135)}) --({cos(225)}, {sin(225)}) --({cos(315)}, {sin(315)});

\draw[thick]
({cos(315)}, {sin(315)}) --({cos(45)}, {sin(45)}) --({cos(135)}, {sin(135)}) --({cos(225)}, {sin(225)}) --({cos(315)}, {sin(315)});
\draw[thick]
({cos(45)}, {sin(45)}) -- ({cos(225)}, {sin(225)});
\end{tikzpicture}
\end{array}
\]
\caption{The $s_2(3) = 10$ ways to 2-color subdivisions of a square}\label{figNgon}
\end{figure}
\end{center}

Other interpretations of $s(n,k)$ (e.g., in terms of standard Young tableaux~\cite{Stanley96}, and loopless outerplanar maps~\cite{GeffnerN17}) can be extended similarly, and it is possible that $s_d(n)$ --- or other generalizations of related integer sequences --- can lend new perspectives to solving combinatorial problems involving these familiar quantities.

\section{Acknowledgments}
\label{sec7}

Finally, we express our gratitude to Neil J.\ A.\ Sloane and the numerous contributors 
to the On-Line Encyclopedia of Integer Sequences for developing and curating 
such a valuable resource for mathematical discovery.

\bigskip
\hrule
\bigskip

\noindent 
2010 \emph{Mathematics Subject Classification}:~Primary 05A15. 
Secondary 05A16, 05A19.

\medskip

\noindent 
\emph{Keywords}:~Schr\"oder number,
Schr\"oder path,
Catalan number,
Dyck path,
Narayana number,
plane rooted tree. 

\bigskip
\hrule
\bigskip

\noindent 
Concerned with sequences
\seqnum{A000108},
\seqnum{A001003},
\seqnum{A006318},
\seqnum{A071684},
\seqnum{A071688},
\seqnum{A086810},
\seqnum{A090181},
\seqnum{A107841},
\seqnum{A131763},
\seqnum{A131765}.
\bigskip
\hrule
\bigskip

\end{document}